\documentclass[a4paper]{article}
\usepackage{hyperref}
\usepackage{graphicx}
\usepackage{mathrsfs}
\usepackage{enumerate} 
\usepackage{amsxtra,amssymb,latexsym, amscd,amsthm}
\usepackage{indentfirst}
\usepackage{color}
\usepackage[utf8]{inputenc}
\usepackage[mathscr]{eucal}
\usepackage{amsfonts}
\usepackage{graphics}
\usepackage{multirow}
\usepackage{array}
\usepackage{subfigure}
\usepackage{cite}
\usepackage{wrapfig}
\usepackage{tikz}
\usepackage{diagbox}

\usepackage{pgfplots}
\pgfplotsset{compat=1.15}
\usepackage{mathrsfs}
\usetikzlibrary{arrows}

\newcommand{\footremember}[2]{%
    \footnote{#2}
    \newcounter{#1}
    \setcounter{#1}{\value{footnote}}%
}
 
\def\R{{\mathbb R}}
\def\C{{\mathbb C}}
\def\N{{\mathbb N}}

\DeclareMathOperator{\rank}{rank}
\DeclareMathOperator{\diag}{diag}

\newtheorem{theorem}{\bf Theorem}
\newtheorem{lemma}{\bf Lemma}
\newtheorem{example}{\bf Example}
\newtheorem{corollary}{\bf Corollary}
\newtheorem{remark}{\bf Remark}

\providecommand{\keywords}[1]
{
  \textbf{\textbf{Keywords:}} #1
}
\begin{document}
\definecolor{qqzzff}{rgb}{0,0.6,1}
\definecolor{ududff}{rgb}{0.30196078431372547,0.30196078431372547,1}
\definecolor{xdxdff}{rgb}{0.49019607843137253,0.49019607843137253,1}
\definecolor{ffzzqq}{rgb}{1,0.6,0}
\definecolor{qqzzqq}{rgb}{0,0.6,0}
\definecolor{ffqqqq}{rgb}{1,0,0}
\definecolor{uuuuuu}{rgb}{0.26666666666666666,0.26666666666666666,0.26666666666666666}
\newcommand{\vi}[1]{\textcolor{blue}{#1}}
\newif\ifcomment
\commentfalse
\commenttrue
\newcommand{\comment}[3]{%
\ifcomment%
	{\color{#1}\bfseries\sffamily#3%
	}%
	\marginpar{\textcolor{#1}{\hspace{3em}\bfseries\sffamily #2}}%
	\else%
	\fi%
}

\newcommand{\mapr}[1]{{{\color{blue}#1}}}
\newcommand{\revise}[1]{{{\color{blue}#1}}}

\title{Exact polynomial optimization strengthened with Fritz John conditions}


\author{%
Ngoc Hoang Anh Mai\footremember{1}{University of Konstanz,
Universit\"atsstra{\ss}e 10, D-78464 Konstanz, Germany.}
  }

\maketitle

\begin{abstract}
Let $f,g_1,\dots,g_m$ be polynomials with real coefficients in a vector of variables $x=(x_1,\dots,x_n)$.
Denote by $\text{diag}(g)$ the diagonal matrix with coefficients $g=(g_1,\dots,g_m)$ and denote by $\nabla g$ the Jacobian of $g$.
Let $C$ be the set of critical points defined by
\begin{equation}
    C=\{x\in\mathbb R^n\,:\,\text{rank}(\varphi(x))< m\}\quad\text{with}\quad\varphi:=\begin{bmatrix}
\nabla g\\
\text{diag}(g)
\end{bmatrix}\,.
\end{equation}
Assume that the image  of $C$ under $f$, denoted by $f(C)$, is empty or finite.
(Our assumption holds generically since $C$ is empty in a Zariski open set in the space of the coefficients of $g_1,\dots,g_m$ with given degrees.)
We provide a sequence of values, which returned by semidefinite programs, finitely converges to the minimal value attained by $f$ over the basic semi-algebraic set $S$ defined by
\begin{equation}
    S:=\{x\in\mathbb R^n\,:\,g_j(x)\ge 0\,,\,j=1,\dots,m\}\,.
\end{equation}
Consequently, we can precisely compute the minimal value of any polynomial with real coefficients in $X$ over one of the following sets: the unit ball, the unit hypercube, and the unit simplex.
Under a slightly more general assumption, we extend this result to the minimization of any polynomial over a basic convex semi-algebraic set that has a non-empty interior and is defined by the inequalities of concave polynomials.
\end{abstract}
\tableofcontents
\section{Introduction}
The study of non-negative polynomials is of interest in real algebraic geometry with applications in polynomial optimization.
In his seminal paper \cite{hilbert1888darstellung}, Hilbert studied the expression of non-negative polynomials as sums of squares of polynomials.
We call Positivstellens\"atze the representations of polynomials positive on a basic semi-algebraic set, a set defined by a system of polynomial inequalities.
Putinar showed in \cite{putinar1993positive} a Positivstellensatz that each polynomial positive on a compact basic semi-algebraic set satisfying the so-called Archimedean condition (stated below) can be decomposed as a linear combination of polynomials defining this basic semi-algebraic set with weights which are sums of squares of polynomials.
Using Putinar's Positivstellensatz, Lasserre introduced in \cite{lasserre2001global} a sequence of values returned by semidefinite programs (also known as Lasserre's hierarchy) to approximate from below as closely as desired the minimal value of a polynomial over a basic semi-algebraic set.

We refer to Nichtnegativstellens\"atze as the representations of polynomials that are non-negative on a basic semi-algebraic set.
They allow us to obtain a sequence of values returned by relaxation programs similar in spirit to Lasserre's hierarchy that converges finitely to the optimal value for a given polynomial optimization problem.
However, not all Nichtnegativstellens\"atze have clearly practical applications.
For instance, the relaxation programs based on Krivine--Stengle's Nichtnegativstellens\"atze \cite{krivine1964anneaux} are not convex, and hence it is hard to obtain the values returned by such programs.
It is because Krivine--Stengle's Nichtnegativstellens\"atze has non-prescribed denominators.
In other words, the corresponding relaxation programs are convex, namely semidefinite programs, if Nichtnegativstellensatz has no denominator or has a prescribed denominator.
We refer the readers to recent Nichtnegativstellens\"atze without denominators by Scheiderer \cite{scheiderer2000sums,scheiderer2003sums,scheiderer2006sums} for some compact basic semi-algebraic sets of low dimensions.
His works involve the non-strict extension of Schm\"udgen's Positivstellensatz \cite{schmudgen1991thek} originally stated that each polynomial positive on a compact basic semi-algebraic set can be written as a linear combination of products of polynomials defining this set with weights which are sums of squares of polynomials.

In this paper we provide some Nichtnegativstellens\"atze without denominators which have the same forms as Putinar's and Schm\"udgen's Positivstellensatz.
To achieve this, we utilize slack variables and additional polynomial equations for a given basic semi-algebraic set.
These polynomial equations, generated by the input polynomials and their gradients, involve the so-called Fritz John conditions.

Let $\R[x]$ denote the ring of polynomials with real coefficients in the vector of variables $x$.
Given $f,g_1,\dots,g_m\in\R[x]$, consider polynomial optimization problem
\begin{equation}\label{eq:pop}
    f^\star:=\inf\limits_{x\in S(g)} f(x)\,,
\end{equation}
where $S(g)$ is the basic semi-algebraic set associated with $g=(g_1,\dots,g_m)$, i.e.,
\begin{equation}
    S(g):=\{x\in\R^n\,:\,g_j(x)\ge 0\,,\,j=1,\dots,m\}\,.
\end{equation}
Given $p\in\R[x]$, we denote by $\nabla p$ the gradient of $p$, i.e., $\nabla p=(\frac{\partial p}{\partial x_1},\dots,\frac{\partial p}{\partial x_n})$.
We state the Fritz John conditions in the following lemma:
\begin{lemma}\label{lem:FJ}
Let $f,g_1,\dots,g_m\in\R[x]$. If $u$ is a  local minimizer for problem \eqref{eq:pop}, then the Fritz John conditions hold for problem \eqref{eq:pop} at $u$, i.e., the following conditions hold:
\begin{equation}
    \begin{cases}
    		\exists (\lambda_0,\dots,\lambda_m)\in [0,\infty)^{m+1}\,:\\
          \lambda_0 \nabla f(u)=\sum_{j=1}^m \lambda_j \nabla g_j(u)\,,\\
          \lambda_j g_j(u) =0\,,\,j=1,\dots,m\,,\\
          \sum_{j=0}^m \lambda_j^2=1
    \end{cases}
   \Leftrightarrow
       \begin{cases}
       \exists (\lambda_0,\dots,\lambda_m)\in \R^{m+1}\,:\\
          \lambda_0^2 \nabla f(u)=\sum_{j=1}^m \lambda_j^2 \nabla g_j(u)\,,\\
          \lambda_j^2 g_j(u) =0\,,\,j=1,\dots,m\,,\\
          \sum_{j=0}^m \lambda_j^2=1\,.
    \end{cases}
\end{equation}
\end{lemma}
In the final conditions of the two sides, the sequence of multipliers $\lambda_j$ is normalized to ensure that the multipliers are not all zeros.

Fritz John derived his conditions in \cite{jhon1948extremum}. 
A proof for Lemma \ref{lem:FJ} can be found in Freund's lecture note \cite[Theorem 10]{freund2004optimality}.

If $\lambda_0>0$, the Fritz John conditions are equivalent to the Karush--Kuhn--Tucker conditions (see \cite{karush1939minima,kuhn1951w}).
Note that there are some cases of problem \eqref{eq:pop} (indicated below) for which the Karush--Kuhn--Tucker conditions do not hold at any global minimizer.

Denote by $\Sigma^2[x]$ the cone of sum of squares of polynomials in $\R[x]$.
Given $g_1,\dots,g_m\in\R[x]$, let $Q(g)[x]$ be the quadratic module associated with $g=(g_1,\dots,g_m)$, i.e.,
\begin{equation}
    Q(g)[x]:=\Sigma^2[x] +\sum_{j=1}^m g_j\Sigma^2[x]\,.
\end{equation}
We say that $S(g)$ satisfies the Archimedean condition if there exists $R>0$ such that $R-x_1^2-\dots-x_n^2\in Q(g)[x]$.
Under the Archimedean condition and some standard optimality conditions (containing the Karush--Kuhn--Tucker conditions), Nie utilizes Marshall's Nichtnegativstellensatz \cite{marshall2009representations,marshall2006representations} to guarantee in \cite{nie2014optimality} finite convergence of Lasserre's hierarchy. 
In this case, the polynomial optimization problem necessarily has finite global minimizers.

Given $g_1,\dots,g_m\in\R[x]$, with $g=(g_1,\dots,g_m)$, let $\Pi g$ be the vector of products of $g_1,\dots,g_m$ defined by
\begin{equation}\label{eq:prod.g}
\Pi g:=(g^\alpha)_{\alpha\in\{0,1\}^m\backslash \{0\}}\,,
\end{equation}
where $\alpha=(\alpha_1,\dots,\alpha_m)$ and $g^\alpha:=g_1^{\alpha_1}\dots g_m^{\alpha_m}$.
We call $Q(\Pi g)[x]$ the preordering generated by $g$, denoted by $P(g)[x]$.
Obviously, if $m=1$, it holds that $P(g)[x]=Q(g)[x]$.

Given $h_1,\dots,h_l\in\R[x]$, let $V(h)$ be the variety defined by $h=(h_1,\dots,h_l)$, i.e.,
\begin{equation}
V(h):=\{x\in\R^n\,:\,h_j(x)=0\,,\,j=1,\dots,l\}\,.
\end{equation}
and let $I(h)[x]$ be the ideal generated by $h$, i.e.,
\begin{equation}
    I(h)[x]:= \sum_{j=1}^l h_j \R[x]\,.
\end{equation}
The real radical of an ideal $I(h)[x]$, denoted by $\sqrt[\R]{I(h)}$, is defined as
\begin{equation}
{\sqrt[\R]{I(h)[x]}}=\{f\in\R[x]\,:\,\exists m\in \N\,:\,-f^{2m}\in\Sigma^2[x]+I(h)[x]\}\,.
\end{equation}
Krivine--Stengle's Nichtnegativstellens\"atze \cite{krivine1964anneaux} imply that
\begin{equation}\label{eq:real.radi}
\sqrt[\R]{I(h)[x]}=\{p\in\R[x]\,:\,p=0\text{ on }V(h)\}\,.
\end{equation}
We say that $I(h)[x]$ is real radical if $I(h)[x]=\sqrt[\R]{I(h)[x]}$.

Demmel, Nie, and Powers provide in \cite{demmel2007representations} a Nichtnegativstellensatz saying that if $f$ is non-negative on a subset of $S(g)$ at which the Karush--Kuhn--Tucker conditions hold for problem \eqref{eq:pop}, then there exists $q\in P(g)[x,\lambda] $ such that $f-q$ vanishes on $V(h_\text{KKT})$, where $\lambda:=(\lambda_1,\dots,\lambda_m)$ and
\begin{equation}
    h_\text{KKT}:=(\nabla f-\sum_{j=1}^m \lambda_j \nabla g_j,\lambda_1g_1,\dots,\lambda_mg_m)\,.
\end{equation}
Here $h_\text{KKT}$ includes polynomials in the Karush--Kuhn--Tucker conditions.
To apply this Nichtnegativstellensatz for exact polynomial optimization, they assume the Karush--Kuhn--Tucker conditions hold at some global minimizer.

Our goal is to deal with the case of problem \eqref{eq:pop} for which the Karush--Kuhn--Tucker conditions do not hold at any global minimizer or the set of global minimizers has a positive dimension.
Given $g_1,\dots,g_m\in\R[x]$, let  $\varphi^g:\R^{n}\to \R^{(n+m)\times m}$ be a function associated with $g=(g_1,\dots,g_m)$ defined by
\begin{equation}
\varphi^g(x)=\begin{bmatrix}
\nabla g(x)\\
\diag(g(x))
\end{bmatrix}=
    \begin{bmatrix}
    \nabla g_1(x)& \dots& \nabla g_m(x)\\
    g_1(x)&\dots&0\\
    .&\dots&.\\
    0&\dots&g_m(x)
    \end{bmatrix}\,.
\end{equation}
Given a real matrix $A$, we denote by $\rank(A)$ the dimension of the vector space generated by the columns of $A$ over $\R$.
We say that a set $\Omega$ is finite if its cardinal number is a non-negative integer.
Let  $C(g)$ be the set of critical points associated with $g$ defined by
\begin{equation*}
    C(g):=\{x\in\R^n\,:\,\rank(\varphi^g(x))< m\}.
\end{equation*}
It is easily seen that $C(g)$ is the set of points at which the Fritz John conditions stated in Lemma \ref{lem:FJ} hold for problem \eqref{eq:pop} in the case of $\lambda_0=0$.
In other words, $C(g)$ is the set of all points at which the Fritz John conditions hold, but the Karush--Kuhn--Tucker conditions do not.
From this, the following lemma follows:
\begin{lemma}\label{lem:empty.Cg}
If $C(g)=\emptyset$, problem \eqref{eq:pop} has no local minimizer or only has local minimizers at which the Karush--Kuhn--Tucker conditions for this problem hold.
\end{lemma}

We state the first main result in the following theorem:
\begin{theorem}\label{theo:rep}
Let $f,g_1,\dots,g_m\in\R[x]$. 
Assume that $f$ is non-negative on $S(g)$  with $g:=(g_1,\dots,g_m)$ and $f(C(g))$ is finite. 
Then there exists $q\in P(g)[x,\bar \lambda]$ such that $f-q$ vanishes on $V(h_\text{FJ})$, where $\bar\lambda:=(\lambda_0,\dots,\lambda_m)$ and
\begin{equation}\label{eq:.polyFJ}
    h_\text{FJ}:=(\lambda_0\nabla f-\sum_{j=1}^m \lambda_j \nabla g_j,\lambda_1g_1,\dots,\lambda_mg_m,1-\sum_{j=0}^m\lambda_j^2)\,.
\end{equation}
Moreover, if $S(g)$ satisfies the Archimedean condition, we can take $q\in Q(g)[x,\bar \lambda]$.
\end{theorem}
Here $h_\text{FJ}$ includes polynomials in the Fritz John conditions stated in Lemma \ref{lem:FJ}.
By \eqref{eq:real.radi}, it is clear that in Theorem \ref{theo:rep} if $I(h_\text{FJ})[x,\bar \lambda]$ is real radical, then $f-q\in I(h_\text{FJ})[x,\bar \lambda]$.
Given $d\in\N$, let $\R[x]_d$ be the set of polynomials of degree at most $d$.

We state the second main result in the following theorem:
\begin{theorem}\label{theo:gen}
Let $d_1,\dots,d_m$ be positive integers. Then there exists a polynomial $\psi$, which is in the coefficients of polynomials $g_j\in\R[x]_{d_j}$ for $j=1,\dots,m$ such that if $\psi$ does not vanish at the input data then $C(g)$ with $g:=(g_1,\dots,g_m)$ is empty.
\end{theorem}
On one hand, Theorem \ref{theo:gen} implies that $C(g)$ is empty, and so is $f(C(g))$ on a Zariski open set in the space of the coefficients of $g_1,\dots,g_m$ with given degrees. 
Thus in Theorem \ref{theo:rep}, our assumption that $f(C(g))$ is finite holds generically.
However, there exists a case of $g=(g_1,\dots,g_m)$ (indicated below) for which both $C(g)$ and $f(C(g))$ are infinite.
To overcome this, we provide in Theorem \ref{theo:rep2} a representation of $f$ with denominator $\lambda_0$ but without assumption on $f(C(g))$.
On the other hand, combining Theorem \ref{theo:gen} and Lemma \ref{lem:empty.Cg} gives the genericity of the Karush--Kuhn--Tucker conditions.

The proofs of Theorems \ref{theo:rep} and \ref{theo:gen} (postponed to Sections \ref{sec:proof.rep} and \ref{sec:proof.gen}) are inspired by the techniques of Demmel--Nie--Power \cite{demmel2007representations} and Nie \cite{nie2014optimality}, respectively. 
To prove Theorem \ref{theo:rep}, we claim that the polynomial $f$ has a finite number of values on the variety $V(h_{FJ})$. 
We prove this by considering $f$ on each connected component of $V(h_{FJ})$ not contained in the hyperplane $\lambda_0=0$ and then applying the mean value theorem.
The remaining case is based on the assumption that $f(C(g))$ is finite. 
The proof of Theorem \ref{theo:gen} relies on the existence of a discriminant for a system of polynomial equations generated by $g_1,\dots,g_m$ under a simple transformation.

Bucero and Mourrain present in \cite[Section 3.3]{bucero2013exact} a variety defined by the Fritz John conditions without giving any representation of polynomials non-negative on semi-algebraic sets in the case of $C(g)\ne\emptyset$.
Note that the polynomial equation $1-\sum_{j=1}^m \lambda_j^2=0$ does not exist in their variety to ensure the non-zero vector of multipliers as in our variety $V(h_\text{FJ})$ with $h_\text{FJ}$ defined as in \eqref{eq:.polyFJ}.

Nie provides in \cite{nie2013exact} a preordering-based representation of polynomial $f$ non-negative on a basic semi-algebraic set $S(g)$ with $g=(g_1,\dots,g_m)$ by adding to this set a large number of polynomial equations generated by the Jacobian of the polynomial map $(f,g_1,\dots,g_m)$. 
To achieve this, he restricts the number of polynomials defining $S(g)$ and assumes that the Jacobian of each subset of $g$ has a full rank on their variety.
For comparison purposes, to obtain our representations in Theorem \ref{theo:rep}, we utilize $n+m+1$ additional polynomial equations (generated by $g_j$ and $\nabla f,\nabla g_j$) for $S(g)$ and $m+1$ slack variables $\lambda_0,\dots,\lambda_m$, which are the multipliers in the Fritz John conditions stated in Lemma \ref{lem:FJ}.
Here we assume that the matrix $\varphi^g(x)$ (generated by $g_j(x)$ and $\nabla g_j(x)$) does not need to have a full rank for each $x\in\R^n$, but it is required that the image of all real points at which $\varphi^g$ is rank-deficient under $f$ is finite.
Under these conditions and the Archimedean condition $S(g)$, we also provide a representation of $f$ involving the quadratic module $Q(g)[x,\bar \lambda]$.

The paper is organized as follows: 
Section \ref{sec:proof} is to prove Theorems \ref{theo:rep} and \ref{theo:gen}. 
We give in Section \ref{sec:examples} some illustrated examples for Theorem \ref{theo:rep}.
A relevant counterexample is indicated in this section.
Section \ref{sec:application} presents the application of our results in computing precisely the optimal value for a polynomial optimization problem. 
Section \ref{sec:variation} shows variations of our main results under slightly more general assumptions.
Section \ref{sec:rep.gen} states the general Nichtnegativstellens\"atze with prescribed denominators based on the Fritz John conditions.

\section{Proof of the main results}
\label{sec:proof}
\subsection{Preliminaries}
In this subsection, we present some preliminaries from algebraic geometry needed for proof of our main results.
We recall one of Krivine--Stengle's Positivstellensatz \cite{krivine1964anneaux} in the following lemma:
\begin{lemma}\label{lem:pos}
Let $g_1,\dots,g_m\in\R[x]$. Assume that $S(g)=\emptyset$ with $g:=(g_1,\dots,g_m)$. 
Then it holds that $-1 \in P(g)[x]$.
\end{lemma}
We recall in the following lemma Putinar's Positivstellensatz  \cite{putinar1993positive}:
\begin{lemma}\label{lem:Pu}
Let $f,g_1,\dots,g_m\in\R[x]$.
Assume that $f$ is positive on $S(g)$ with $g:=(g_1,\dots,g_m)$, and $S(g)\ne \emptyset$ satisfies the Archimedean condition. 
Then it holds that $f \in Q(g)[x]$.
\end{lemma}
Denote by $\deg(p)$ the degree of a given polynomial $p\in\R[x]$.
The following lemma is a consequence of Lemmas \ref{lem:pos} and \ref{lem:Pu}:
\begin{lemma}\label{lem:rep-1.Pu}
Let $g_1,\dots,g_m\in\R[x]$. 
Assume that $g_m:=R-x_1^2-\dots-x_n^2$ for some $R>0$ and $S(g)=\emptyset$ with $g:=(g_1,\dots,g_m)$. 
Then it holds that $-1 \in Q(g)[x]$.
\end{lemma}
\begin{proof}
Since $S(g)=\emptyset$, Lemma \ref{lem:pos} yields that there exists $\sigma_\alpha\in \Sigma^2[x]$ such that 
\begin{equation}\label{eq:rep.-1}
    -1=\sum_{\alpha\in\{0,1\}^n}\sigma_\alpha g^\alpha\,.
\end{equation}
Given $p\in\R[x]$ with $u=\lceil \deg(p)/2\rceil$, let $\bar p:=x_0^{2u}p(x/x_0) \in\R[\bar x]$, where $\bar x:=(x_0,x)$.
For instance, $\bar g_m=Rx_0^2-x_1^2-\dots-x_n^2$.
Let $d$ be an integer number such that $2d\ge \deg(\sigma_\alpha g^\alpha)$.
Let $d_j=\lceil\deg(g_j)/2\rceil$.
From \eqref{eq:rep.-1}, we get
\begin{equation}\label{eq:equi}
    -x_0^{2d}=\sum_{\alpha\in\{0,1\}^n}\psi_\alpha \bar g^\alpha\,,
\end{equation}
where $\psi_\alpha=x_0^{2(d-d_j)}\sigma_\alpha(x/x_0)\in\Sigma^2[\bar x]$ and $\bar g=(\bar g_1,\dots,\bar g_m)$.
Denote by $w\in\R[\bar x]$ the polynomial on the right-hand side of \eqref{eq:equi}.
Then $w$ is non-negative on $S(\bar g,1-x_0^2)$.
Since $0\in S(\bar g,1-x_0^2)$, we get $S(\bar g,1-x_0^2)\ne \emptyset$. 
On the other hand $S(\bar g,1-x_0^2)$ satisfies the Archimedean condition since
\begin{equation}
    (R+1)-x_0^2-\dots-x_n^2=(R+1)(1-x_0^2)+\bar g_m\in Q(\bar g,1-x_0^2)[\bar x]\,.
\end{equation}
Applying Lemma \ref{lem:Pu}, we obtain
\begin{equation}
    -x_0^{2d}+\frac{1}{2}=w+\frac{1}{2}\in Q(\bar g,1-x_0^2)[\bar x]\,.
\end{equation}
Letting $x_0=1$ implies that $-\frac{1}{2}\in Q(g)[x]$, yielding the result.
\end{proof}
In the following lemma, we obtain the same result as Lemma \ref{lem:rep-1.Pu} under a weaker condition:
\begin{lemma}\label{lem:rep-1.Pu2}
Let $g_1,\dots,g_m\in\R[x]$ such that $S(g)$ with $g:=(g_1,\dots,g_m)$ satisfies the Archimedean condition and $S(g)=\emptyset$. 
Then $-1 \in Q(g)[x]$.
\end{lemma}
\begin{proof}
Since $S(g)$ satisfies the Archimedean condition, there exists $R>0$ such that $g_{m+1}:=R-x_1^2-\dots-x_n^2\in Q(g)[x]$.
It implies that $S(g)\subset S(g_{m+1})$, which gives $S(g,g_{m+1})=S(g)=\emptyset$.
By using Lemma \ref{lem:rep-1.Pu}, we obtain $-1\in Q(g,g_{m+1})[x]\subset Q(g)[x]$, yielding the result.
\end{proof}
Given $h_1,\dots,h_l\subset\R[x]$, let $V_\C(h)$ be the complex variety defined by $h=(h_1,\dots,h_m)$, i.e.,
\begin{equation}
    V_\C(h):=\{x\in\C^n\,:\,h_j(x)=0\,,\,j=1,\dots,l\}\,.
\end{equation}
Denote by $\delta_{ij}$ the Kronecker delta function at $(i,j)\in\N^2$.

The following lemma is a direct consequence of \cite[Lemma 2.4 and Remark 2.5]{demmel2007representations}:
\begin{lemma}\label{lem:interpolate}
Let $U_1,\dots,U_r$ be pairwise disjoint complex varieties defined by finitely many polynomials in $\R[x]$. 
Then there exist polynomials $p_1,\dots,p_r\in \R[x]$ such that $p_i(U_j)=\delta_{ij}$.
\end{lemma}

We generalize the definition of a basic semi-algebraic set as follows: A semi-algebraic subset of R is a subset of the form
\begin{equation}\label{eq:def.semi.set}
\bigcup_{i=1}^t\bigcap_{j=1}^{r_i}\{x\in\R^n\,:\,f_{ij}(x)*_{ij}0\}\,,
\end{equation}
where $f_{ij}\in\R[x]$ and $*_{ij}$ is either $>$ or $=$.
Given two semi-algebraic sets $A\subset \R^n$ and $B\subset \R^m$, we say that a mapping $f : A \to B$ is semi-algebraic if its graph is a semi-algebraic set in $\R^{n+m}$.

The following lemma can be found in \cite[Proposition 1.6.2 (ii)]{pham2016genericity}:
\begin{lemma}\label{lem:composit.semi-al}
Compositions of semi-algebraic maps are semi-algebraic.
\end{lemma}

A semi-algebraic subset $A$ is said to be semi-algebraically path connected if, for every $x,y$ in $A$, there exists a continuous semi-algebraic mapping $\phi:[0,1] \to A$ such that $\phi(0) = x$ and $\phi(1) = y$.

Combining \cite[Theorem 2.4.5 and Proposition 2.5.13]{bochnak2013real}, we obtain the following lemma:
\begin{lemma}\label{eq:component.semi}
Every semi-algebraic set has a finite number of components, which are semi-algebraically path connected.
\end{lemma}
The following result is a direct consequence of Lemma \ref{eq:component.semi} since the difference between two real varieties is a semi-algebraic set according to the definition \eqref{eq:def.semi.set}:
\begin{lemma}\label{lem:diff.connect}
The difference between two real varieties has a finite number of components, which are semi-algebraically path connected.
\end{lemma}
\if{
\begin{lemma}
Let 
\end{lemma}
}\fi

We state in the following lemma a decomposition of the intersection of a given complex variety with a real space:
\begin{lemma}\label{lem:quadra}
Let $f\in\R[x]$ and let $W$ be a complex variety defined by finitely many polynomials in $\R[x]$. 
Assume that $f(W\cap \R^n)$ is finite.
Then there exists a finite sequence of subsets $W_1,\dots,W_r$ such that the following conditions hold:
\begin{enumerate}
\item $W_1,\dots,W_r$ are pairwise disjoint complex varieties defined by finitely many polynomials in $\R[x]$;
\item for $j=1,\dots,r$, $W_j\subset W$, and  $f$ is constant on $W_j$;
\item $(W_1\cup\dots\cup W_r)\cap \R^{n}=W\cap \R^{n}$.
\end{enumerate}
\end{lemma}
\begin{proof}
By assumption, we get
$f(W\cap \R^{n}) = \{t_1 ,\dots, t_r \} \subset \R$,
where $t_i\ne t_j$ if $i\ne j$.
For $j=1,\dots,r$, let
$W_j:=W\cap V_\C(f-t_j)$.
Then $W_j$ is a complex variety defined by finitely many polynomials in $\R[x]$.
It is clear that $f(W_j)=\{t_j\}$.
We claim that $W_1,\dots,W_r$ are pairwise disjoint. 
Otherwise, let $x\in W_i\cap W_j$ with $i\ne j$. 
It implies that $t_i=f(x)=t_j$ which is impossible.
Let $U=W_1\cup\dots\cup W_r$.
We show that $W\cap \R^{n}=U\cap \R^{n}$.
Let $x\in W\cap \R^{n}$. 
Then there is $j\in\{1,\dots,r\}$ such that $f(x)=t_j$ which gives $x\in W_j\subset U$ so we get $x\in U\cap \R^{n}$. 
Thus $W\cap \R^{n}\subset U\cap \R^{n}$ since $x$ is arbitrary in $W\cap \R^{n}$.
Conversely, suppose that $x\in U\cap \R^{n}$. 
By the definition of $U$, there is $j\in\{1,\dots,r\}$ such that $x\in W_j$.
It implies that $x\in W$ by definition of $W_j$. 
Then $x\in W\cap \R^{n}$.
Thus $U\cap \R^{n}\subset W\cap \R^{n}$ since $x$ is arbitrary in $U\cap \R^{n}$.
\end{proof}
We use the technique from the proof of  \cite[Theorem 3.2]{demmel2007representations} to obtain the following lemma:
\begin{lemma}\label{lem:const.func}
Let $f,g_1,\dots,g_m\in\R[x]$. 
Assume that $f$ is non-negative on $S(g)$ with $g=(g_1,\dots,g_m)$. 
Let $U_1,\dots,U_r$ be pairwise disjoint complex varieties defined by finitely many polynomials in $\R[x]$.
Set $U=U_1\cup\dots\cup U_r$.
Assume that $f$ is constant on each $U_i$.
Then there exists $q\in P(g)[x]$ such that $f-q$ vanishes on $U\cap\R^n$.
Moreover, if $S(g)$ satisfies the Archimedean condition, we can take $q\in Q(g)[x]$.
\end{lemma}
\begin{proof}
Let $W_0$ be the union of all $U_j$ whose intersection with $S(g)$ is empty. 
Then $W_0$ is a complex variety defined by finitely many polynomials in $\R[x]$.
Let $W_1,\dots,W_r$ be the remaining $U_j$. 
Thus $f$ is constant on $W_j$, for $j=1,\dots,r$.
Further, since $f$ is non-negative on the non-empty set $S(g)\cap W_j$, there exists $\alpha_j>0$ such that $f=\alpha_j$ on $W_j$, $j=1,\dots,r$.
Set $q_j(x) = \alpha_j\in\Sigma^2[x]$ then we get $f = q_j$ on $W_j$.
Observe that $U= W_0 \cup W_1 \cup \dots \cup W_r$, where $W_0,\dots,W_r$ are pairwise disjoint.
By Lemma \ref{lem:interpolate}, there exist polynomials
$p_0,p_1,\dots,p_r \in \R[x]$ such that $p_i(W_j) = \delta_{ij}$.
By assumption, it holds that $W_0 \cap S(g) = \emptyset$ and hence by Theorem \ref{lem:pos}, there exists $v_0 \in P(g)[x]$ such that
$-1 = v_0$ on $W_0\cap \R^{n}$. 
We have $f = s_1 - s_2$ for the SOS polynomials $s_1 =(f+\frac{1}{2})^2$ and $s_2 = f^2+\frac{1}{4}$.
It implies that $f = s_1 + v_0 s_2$ on $W_0\cap \R^{n}$.
Let $q_0 = s_1 +v_0 s_2 \in P(g)[x]$.  
Now let $q =\sum_{i=0}^r q_i p_i^2$ then $q\in P(g)[x]$ and we obtain $f - q$ vanishes on $U\cap \R^n$.
Assume that $S(g)$ satisfies the Archimedean condition. 
Following Lemma \ref{lem:rep-1.Pu2}, we can take $v_0\in Q(g)[x]$, which implies that $q_0$ is in $Q(g)[x]$ then so is $q$.
\end{proof}

\subsection{Proof of the representations}
\label{sec:proof.rep}
Recall the vector of variables $\bar\lambda:=(\lambda_0,\dots,\lambda_m)$. 
For simplicity of notation, set
\begin{equation}
    \{\lambda_0=0\}:=\{(x,\bar\lambda)\in\mathbb R^{n+m+1}\,:\,\lambda_0=0\}\,.
\end{equation}
Let $\pi:\R^{n+m+1}\to \R^n$ be the projection defined by 
\begin{equation}
    \pi(x,\bar\lambda)=x\,,\,\forall x\in\R^n\,,\,\forall \bar\lambda\in\R^{m+1}\,.
\end{equation}
We characterize the set of critical points in the following lemma:
\begin{lemma}\label{lem:equa}
Let $f,g_1,\dots,g_m\in\R[x]$.
Let $h_\text{FJ}$ be defined as in \eqref{eq:.polyFJ}.
Set $g:=(g_1,\dots,g_m)$.
Then it holds that $C(g)=\pi(V(h_\text{FJ})\cap \{\lambda_0=0\})$.
\end{lemma}
\begin{proof}
The result follows thanks to the following equivalences:
\begin{equation}
    \begin{array}{rl}
         &  x\in C(g)\\
        \Leftrightarrow & \rank(\varphi^g(x)) < m\\
        \Leftrightarrow & \exists \lambda\in\R^{m}\,:\, \sum_{j=1}^m\lambda_j^2=1\,,\, \sum_{j=1}^m \lambda_j \nabla g_j(x)=0\,,\, \lambda_j g_j(x) =0\\
        \Leftrightarrow & \exists \bar \lambda\in\R^{m+1}\,:\,
        \sum_{j=0}^m\lambda_j^2=1\,,\,\lambda_0=0\,,\,\lambda_0\nabla f(x)=\sum_{j=1}^m \lambda_j \nabla g_j(x)\,,\, \lambda_j g_j(x) =0\\
        \Leftrightarrow & \exists \bar \lambda\in\R^{m+1}\,:\,(x,\bar\lambda)\in V(h_\text{FJ})\cap \{\lambda_0=0\}\\
        \Leftrightarrow& x\in\pi(V(h_\text{FJ})\cap \{\lambda_0=0\})\,.
    \end{array}
\end{equation}
\end{proof}
The following lemma is given in \cite[Theorem 1.8.1]{pham2016genericity}:
\begin{lemma}\label{lem:semial.func.anal}
Let $f:(a, b)\to \R$ be a semi-algebraic
function. Then there are $a = a_0 < a_1 < \dots < a_s < a_{s+1} = b$ such that, for each $i = 0,\dots, s$, the restriction $f|_{(a_i,a_{i+1})}$ is analytic.
\end{lemma}
The following lemma follows from the mean value theorem:
\begin{lemma}\label{lem:mean.val}
Let $f:[0,1]\to \R$ be a continuous piecewise-differentiable function, i.e., there exist $0=a_1<\dots<a_r=1$ such that $f$ is continuous, and $f$ is differentiable on each open interval $(a_i,a_{i+1})$.
Assume that $f$ has zero subgradient.
Then $f(0)=f(1)$.
\end{lemma}
\begin{proof}
By using the mean value theorem on each open interval $(a_i,a_{i+1})$, we get $f(a_i)=f(a_{i+1})$.
Hence $f(0)=f(a_1)=\dots=f(a_{r})=f(1)$ yields the result.
\end{proof}

The following lemma extends \cite[Lemma 3.3]{demmel2007representations} to the case of varieties defined by the Fritz John conditions:
\begin{lemma}\label{lem:constant}
Let $f,g_1,\dots,g_m\in\R[x]$. 
Assume that $f(C(g))$ with $g:=(g_1,\dots,g_m)$ is finite.
Let $h_\text{FJ}$ be defined as in \eqref{eq:.polyFJ}.
Let $W$ be a semi-algebraically path connected component of $V(h_\text{FJ})$. Then $f$ is constant on $W$.
\end{lemma}
\begin{proof} 
Let $(x^{(0)},\bar \lambda^{(0)})$ and $(x^{(1)},\bar \lambda^{(1)})$ in $W$. 
We claim that $f(x^{(0)}) = f(x^{(1)})$.
By assumption, there exists a continuous piecewise-differentiable path $\phi(\tau) = (x(\tau), \bar \lambda(\tau))$, for $\tau\in[0,1]$, lying inside $W$ such that $\phi(0) = (x^{(0)},\bar \lambda^{(0)})$ and $\phi(1) = (x^{(1)},\bar \lambda^{(1)})$ (see, e.g., \cite[Theorem 1.8.1]{pham2016genericity}).
Since $\tau\mapsto\lambda_0(\tau)$ is continuous on $[0,1]$, the set $\lambda_0^{-1}(0)\subset [0,1]$ is closed.
Then there exists a sequence of intervals $[a_j,b_j]\subset[0,1]$ for $j=1,\dots,r$ such that
\begin{itemize}
    \item $0\le a_1\le b_1< a_2\le b_2<\dots< a_r\le b_r\le 1$;
    \item $\forall \tau\in[0,1]\,,\,\lambda_0(\tau)=0 \Leftrightarrow \tau\in\cup_{j=1}^r[a_j,b_j]$.
\end{itemize}
We claim that $\tau\mapsto f(x(\tau))$ is constant on $[a_j,b_j]$  for $j=1,\dots,r$. 
Let $j\in\{1,\dots,r\}$ be fixed. 
Assume by contradiction that there are $\tau_1,\tau_2\in [a_j,b_j]$ such that $\tau_1<\tau_2$ and $f(x(\tau_1))\ne f(x(\tau_2))$.
Since $\tau\mapsto f(x(\tau))$ is continuous on $[a_j,b_j]$, the set $f(x([\tau_1,\tau_2]))$ is infinite. 
This contradicts the assumption $f(C(g))=f(\pi(V(h_\text{FJ})\cap \{\lambda_0=0\}))$ (according to Lemma \ref{lem:equa})
is finite since $f(\pi(V(h_\text{FJ})\cap \{\lambda_0=0\}))\supset f(x([\tau_1,\tau_2]))$ which is due to the fact that
\begin{equation}
    \pi(V(h_\text{FJ})\cap \{\lambda_0=0\})\supset \pi(W\cap \{\lambda_0=0\}) \supset x([\tau_1,\tau_2])\,.
\end{equation}

On the other hand $\tau\mapsto\lambda_0(\tau)$ has no zero value on each of the following open intervals:
\begin{equation}\label{eq:intervals.open}
    (0,a_1),(b_1,a_2),\dots,(b_{r-1},a_r),(b_r,1)
\end{equation}
We claim that $\tau\mapsto f(x(\tau))$ is constant on each of open intervals in \eqref{eq:intervals.open}.

Let us prove this for the interval $(0,a_1)$. 
The proof  for the other intervals is similar.
Let $\tau_1,\tau_2$ be arbitrary in $(0,a_1)$.
The Lagrangian function
\begin{equation}\label{eq:Lagran}
    L(x,\bar \lambda) = f(x)-\sum_{j=1}^m \frac{\lambda_j}{\lambda_0} g_j (x)\,.
\end{equation}
is equal to $f(x)$ on $V(h_\text{FJ})\backslash \{\lambda_0=0\}$, which contains $\phi([\tau_1,\tau_2])$. 
By Lemma \ref{lem:composit.semi-al}, the function $L\circ \phi$ is semi-algebraic.
Moreover, the function $L\circ \phi$ is continuous since $L$ and $\phi$ are continuous.
It implies that $L\circ \phi$ is a continuous piecewise-differentiable function thanks to Lemma \ref{lem:semial.func.anal}.
Note that the function
$L\circ \phi$
has zero subgradient on $[\tau_1,\tau_2]$.
From Lemma \ref{lem:mean.val}, it follows that $f(x (\tau_1) )=(L\circ \phi)(\tau_1)= (L\circ \phi)(\tau_2)= f(x (\tau_2) )$.
We now obtain $f(x(\tau_1))$ = $f(x(\tau_2))$.

By its continuity, $\tau\mapsto f(x(\tau))$ is  constant on the following closed intervals: \begin{equation}\label{eq:intervals}
    [0,a_1],[b_1,a_2],\dots,[b_{r-1},a_r],[b_r,1]
\end{equation}
It implies that $\tau\mapsto f(x(\tau))$ is  constant on $[0,1]$ since it is constant on each $[a_j,b_j]$.
We now obtain $f (x^{(0)})$ = $f (x^{(1)})$, and hence $f$ is constant on $W$.
\end{proof}

\subsubsection*{Proof of Theorem \ref{theo:rep}}
\begin{proof}
Using Lemma \ref{eq:component.semi}, we decompose $V(h_\text{FJ})$ into semi-algebraically path connected components:
$Z_1,\dots,Z_s$.
Lemma \ref{lem:constant} yields that $f$ is constant on $Z_i$.
Thus $f(V(h_\text{FJ}))$ is finite.
Observe that $V(h_\text{FJ})=V_\C(h_\text{FJ})\cap\R^{n+m+1}$.
By using Lemma \ref{lem:quadra}, we obtain a finite sequence of subsets $W_{1},\dots,W_{r}$ such that the following conditions hold:
\begin{itemize}
\item $W_{1},\dots,W_{r}$ are pairwise disjoint complex varieties defined by finitely many polynomials in $\R[x,\bar\lambda]$;
\item for $j=1,\dots,r$, $W_{j}\subset V_\C(h_\text{FJ})$, and  $f$ is constant on $W_{j}$;
\item $(W_{1}\cup\dots\cup W_{r})\cap \R^{n+m+1}=V_\C(h_\text{FJ})\cap \R^{n+m+1}$.
\end{itemize}
Let $D$ be the union of $W_{1},\dots,W_{r}$.
Let $b=1-\lambda_0^2-\dots-\lambda_m^2$.
From this, Lemma \ref{lem:const.func} yields that there exists $p\in P(g,b)[x,\bar\lambda]$ such that $f-p$ vanishes on $D\cap\R^{n+m+1}=V(h_\text{FJ})$.
We write 
\begin{equation}\label{eq:preo.rep.b}
p=\sum_{\alpha\in\{0,1\}^m}\sigma_\alpha g^\alpha+b \sum_{\beta\in\{0,1\}^m}\psi_\beta g^\beta\,,
\end{equation}
for some $\sigma_\alpha,\psi_\beta\in\Sigma^2[x,\bar{\lambda}]$.
Let 
$q=\sum_{\alpha\in\{0,1\}^m}\sigma_\alpha g^\alpha\in P(g)[x,\bar\lambda]$.
Since $b=0$ on $V(h_\text{FJ})$, it holds that $f=p=q$ on $V(h_\text{FJ})$.

Assume that $S(g)$ satisfies the Archimedean condition. Then there exists $R>0$ such that $g_{m+1}=R-x_1^2-\dots-x_n^2\in Q(g)[x]$.
It implies that $S(g,b)$ with $b=1-\lambda_0^2-\dots-\lambda_m^2$ satisfies the Archimedean condition.
It is because
\begin{equation}\label{eq:large.ball}
(R+1)-x_1^2-\dots-x_n^2-\lambda_0^2-\dots-\lambda_m^2=b+g_{m+1}\in Q(g,b)[x,\bar\lambda]\,.
\end{equation}
From this, Lemma \ref{lem:const.func} shows that there exists $p\in Q(g,b)[x,\bar\lambda]$ such that $f-p$ vanishes on $D\cap\R^{n+m+1}=V(h_\text{FJ})$.
We write 
\begin{equation}\label{eq:quadra.rep.b}
p=\sigma_0+\sum_{j=1}^m\sigma_j g_j+b\sigma_{m+1}\,,
\end{equation}
for some $\sigma_j\in\Sigma^2[x,\bar{\lambda}]$.
Let 
$q=\sigma_0+\sum_{j=1}^m\sigma_j g_j\in Q(g)[x,\bar\lambda]$.
Since $b=0$ on $V(h_\text{FJ})$, $f=p=q$ on $V(h_\text{FJ})$.
This completes the proof. 
\end{proof}
\subsection{Proof of the genericity of the assumption}
\label{sec:proof.gen}
Given $p=(p_1,\dots,p_m)$ with $p_1,\dots,p_m\in\R[x]$, we denote the Jacobian of $p$ by $\nabla p=[\nabla p_1\dots \nabla p_m]$.
Given a complex matrix $A$, we denote by $\rank_\C(A)$  the dimension of the vector space generated by the columns of $A$ over $\C$.
Note that $\rank(A)$ is the dimension of the vector space generated by the columns of $A$ over $\R$.

In the following lemma, we recall Nie's result in \cite{nie2012discriminants} involving discriminants:
\begin{lemma}\label{lem:discri}
Let $p_1,\dots, p_m\in\R[x]$ be of degrees $d_1,\dots,d_m$, respectively, and $m \le n$. 
Suppose at least one $d_i > 1$.  
Then there exists a polynomial $\psi$ in the coefficients of $p_1,\dots,p_m$ having the property that with $p:=(p_1,\dots, p_m)$, $\psi(p) = 0$ if and only if there exists
$u \in \C^n$ satisfying
\begin{equation}
p_1(u) =\dots= p_m(u) = 0\quad\text{ and }\quad\rank_\C(\nabla p(u)) < m\,.
\end{equation}
\end{lemma}
Polynomial $\psi$ in Lemma \ref{lem:discri} is called a discriminant of $p$, denoted by $\Delta(p)$.
\subsubsection*{Proof of Theorem \ref{theo:gen}}
\begin{proof}
Let $y=(y_1,\dots,y_m)$.
Let $p_j(x,y)=g_j(x)-y_j^2\in\R[x,y]$. 
Then $\deg(p_j)>1$.
With $p=(p_1,\dots,p_m)$, we get
\begin{equation}
\nabla p(x,y)=\begin{bmatrix}
    \nabla g_1(x)& \dots& \nabla g_m(x)\\
    -2y_1&\dots&0\\
    .&\dots&.\\
    0&\dots&-2y_m
    \end{bmatrix}\,.
\end{equation}
Let $\psi(g):=\Delta(p)$.
Then $\psi$ is a polynomial in the coefficients of $g_1,\dots,g_m$.
We make the following implications:
\begin{equation}
\begin{array}{rl}
&C(g)\ne \emptyset\\
\Rightarrow& \exists x\in C(g)\\
\Rightarrow& \exists x\in\R^n\,:\,\rank(\varphi^g(x))<m\\
\Rightarrow& \exists x\in\R^n\,:\,\exists\lambda\in \R^{m}\backslash\{0\}\,:\,\sum_{j=1}^m\lambda_j\nabla g_j(x)=0\,,\,\lambda_jg_j(x)=0\\
\Rightarrow& \exists x\in\C^n\,:\,\exists\lambda\in \C^{m}\backslash\{0\}\,:\,\sum_{j=1}^m\lambda_j\nabla g_j(x)=0\,,\,\lambda_jg_j(x)=0\\
\Rightarrow& \exists (x,y) \in \C^{n+m}\,:\,g_j(x)=y_j^2\,,\,\exists \lambda\in \C^m \backslash\{0\}\,:\,\sum_{j=1}^m \lambda_j\nabla g_j(x)=0\,,\,-2\lambda_jy_j=0\\
\Rightarrow& \exists (x,y) \in \C^{n+m}\,:\,p_j(x,y)=0\,,\,\rank_\C(\nabla p(x,y)) < m\\
\Rightarrow& \Delta(p)= 0\qquad\text{(according to Lemma \ref{lem:discri})}\\
\Rightarrow& \psi(g)= 0\,.
\end{array}
\end{equation}
Thus it holds that if $\psi(g)\ne 0$, then $ C(g)=\emptyset$, yielding the result.
\end{proof}
\if{
\begin{remark}
Assume that $C(g)=\emptyset$. Let $x$ be a global minimizer of problem \eqref{eq:pop}. Then KKT condition holds at $x$.
\end{remark}
}\fi

\section{Illustrated examples}
\label{sec:examples}
In this section, we illustrate our Nichtnegativstellens\"atze stated in Theorem \ref{theo:rep} with several explicit examples that belong to the following cases of problem \eqref{eq:pop}:
\begin{itemize}
\item unconstrained case (see Example \ref{exam:unc}) and constrained case (see Examples \ref{exam:1} \ref{exam:2}, \ref{exam:infinite.critical}, \ref{exam:inf.minimizer}, \ref{exam:ball}, \ref{exam:box}, \ref{exam:simplex}, \ref{exam:not.attain}, \ref{exam:counterexample});
\item the Karush--Kuhn--Tucker conditions do not hold at any global minimizer (see Examples \ref{exam:1}, \ref{exam:2}, \ref{exam:infinite.critical}, \ref{exam:inf.minimizer});
\item the basic semi-algebraic set $S(g)$ is non-compact (see Example \ref{exam:unc}, \ref{exam:2}, \ref{exam:infinite.critical}, \ref{exam:inf.minimizer}) and the set of global minimizers is infinite
(see Example \ref{exam:inf.minimizer});
\item both the set of critical points $C(g)$ and its image under $f$ are empty (see Examples \ref{exam:ball}, \ref{exam:box}, \ref{exam:simplex}, \ref{exam:not.attain});
\item both the set of critical points $C(g)$ and its image under $f$ are finite (see Examples \ref{exam:1}, \ref{exam:2});
\item neither the set of critical points $C(g)$ nor its image under $f$ is finite (see Example \ref{exam:infinite.critical});
\item the set of critical points $C(g)$ is infinite but its image under $f$ is finite (see Example \ref{exam:inf.minimizer});
\item the minimal value $f^\star$ is not attained (see Example \ref{exam:not.attain});
\item the set $f(C)$ is infinite and there does not exist $q\in P(g)[x,\bar \lambda]$ such that $f-q$ vanishes on $V(h_\text{FJ})$ with $h_\text{FJ}$ defined as in \eqref{eq:.polyFJ} (see Example \ref{exam:counterexample});
\item the set $f(C)$ is infinite but there exists $q\in P(g)[x,\bar \lambda]$ such that $f-q$ vanishes on $V(h_\text{FJ})$ (see Example \ref{exam:inf.minimizer}).
\end{itemize}
We begin with the following example in the case where the basic semi-algebraic set is the whole space:
\begin{example}\label{exam:unc}
Consider the unconstrained case of problem \eqref{eq:pop}.
We can assume that $f^\star=0$, $m=1$, and $g=(g_1)=(1)$. It is obvious that $\rank(\varphi^g(x))=1=m$, which implies $C(g)$ is empty and so is $f(C(g))$. 
Using Theorem \ref{theo:rep}, we get $q\in\Sigma^2[x,\lambda_0,\lambda_1]$ such that $f-q$ vanishes on $V(h_{FJ})$ with $h_{FJ}$ defined as in  \eqref{eq:.polyFJ}, i.e.,
\begin{equation}
    h_\text{FJ}:=(\lambda_0\nabla f,\lambda_1,1-\lambda_0^2-\lambda_1^2)\,.
\end{equation}
With $\lambda_0=1$ and $\lambda_1=0$, $q(x,1,0)\in\Sigma^2[x]$, and $f-q(x,1,0)$ vanishes on $V(\nabla f)$. 
It is exactly the same as the result of Nie, Demmel and Surmfels in \cite{nie2006minimizing}.
\end{example}
In the following four examples, $f^\star=0$ is attained, but the Karush--Kuhn--Tucker conditions do not hold at any global minimizer for problem \eqref{eq:pop}:
\begin{example}\label{exam:1}
Consider the problem \eqref{eq:pop} with $n=m=1$, $f=x$, and $g=(g_1)=(-x^2)$. 
Then $0$ is the unique global minimizer for this problem.
The condition $\rank(\varphi^g(x))< m$ can be expressed as
\begin{equation}
\exists \lambda_1\in \R\backslash\{0\}\,:\,\lambda_1\nabla g_1(x)=-2\lambda_1 x=0\,,\,\lambda_1g_1(x)=-\lambda_1 x^2=0\,.
\end{equation}
It is equivalent to $x=0$. 
Thus $C(g)=\{0\}$ is singleton, and so is $f(C(g))$.
Since $1-x^2=1+g_1\in Q(g)[x]$, $S(g)$ satisfies the Archimedean condition. 
From this, Theorem \ref{theo:rep} yields that there exists $q\in Q(g)[x,\bar \lambda]$ with $\bar\lambda=(\lambda_0,\lambda_1)$ such that $f-q$ vanishes on $V(h_\text{FJ})$ with $h_\text{FJ}$ defined as in \eqref{eq:.polyFJ}, namely $h_\text{FJ}=(\lambda_0+2\lambda_1x,-\lambda_1 x^2,1-\lambda_0^2-\lambda_1^2)$.
It is easy to check that $V(h_\text{FJ})=\{(0,0,\pm 1)\}$.
It implies the selection $q=0$ in this case.
\end{example}
The following example involves the case where the basic semi-algebraic set is non-compact:
\begin{example}\label{exam:2}
Consider the problem \eqref{eq:pop} with $n=2$, $m=1$, $f=(x_1+1)^2+x_2^2-1$, and $g=(g_1)=(x_1^3-x_2^2)$. 
Then $(0,0)$ is the unique global minimizer for this problem.
The condition $\rank(\varphi^g(x))< m$ can be expressed as
\begin{equation}
\exists \lambda_1\in \R\backslash\{0\}\,:\,\lambda_1\nabla g_1(x)=\lambda_1\begin{bmatrix}
3x_1^2\\
-2x_2
\end{bmatrix}=0\,,\,\lambda_1g_1(x)=\lambda_1 (x_1^3-x_2^2)=0\,.
\end{equation}
It is equivalent to $x=(x_1,x_2)=0$. 
Thus $C(g)=\{0\}$ is singleton then so is $f(C(g))$.
From this, Theorem \ref{theo:rep} yields that there exists $q\in P(g)[x,\bar \lambda]$ with $\bar\lambda=(\lambda_0,\lambda_1)$ such that $f-q$ vanishes on $V(h_\text{FJ})$ with $h_\text{FJ}$ defined as in \eqref{eq:.polyFJ}, i.e.,
\begin{equation}
h_\text{FJ}=(\lambda_0\begin{bmatrix}
2(x_1+1)\\
2x_2
\end{bmatrix}-\lambda_1\begin{bmatrix}
3x_1^2\\
-2x_2
\end{bmatrix},\lambda_1 (x_1^3-x_2^2),1-\lambda_0^2-\lambda_1^2)\,.
\end{equation}
Let $(x,\bar\lambda)\in V(h_\text{FJ})$. 
We get $2(\lambda_0+\lambda_1)x_2=0$, so either $x_2=0$ or $\lambda_0+\lambda_1=0$.
If $\lambda_0+\lambda_1=0$, then we obtain $\lambda_0=-\lambda_1\in\{\pm \frac{1}{\sqrt{2}}\}$ and $2(x_1+1)+3x_1^2=0$, which has no real solution, and hence this is impossible.
Consequently $x_2=0$ from which we obtain $\lambda_1x_1^3=0$ which gives $\lambda_1=0$ or $x_1=0$. 
If $\lambda_1=0$, we obtain $\lambda_0\in\{\pm 1\}$ which implies $x_1=-1$.
If $x_1=0$, we get $\lambda_0=0$ which implies $\lambda_1\in\{\pm 1\}$.
These give $V(h_\text{FJ})=\{(-1,0,\pm 1,0),(0,0,0,\pm 1)\}$.
It is not hart to check that $q=g_1$ satisfies $q\in  P(g)[x,\bar \lambda]$ and $f-q$ vanishes on $V(h_\text{FJ})$.
\end{example}

In the last two examples, the sets of critical points are finite. 
We indicate in the following two examples the case where the sets of critical points are infinite:
\begin{example}\label{exam:infinite.critical}
Consider the problem \eqref{eq:pop} with $n=2$, $m=1$, $f=x_1+x_2^2$, and $g=(g_1)=(-x_1^2)$. 
Then $(0,0)$ is the unique global minimizer.
The condition $\rank(\varphi^g(x))< m$ can be expressed as
\begin{equation}
\exists \lambda_1\in \R\backslash\{0\}\,:\,\lambda_1\nabla g_1(x)=\lambda_1\begin{bmatrix}
-2x_1\\
0
\end{bmatrix}=0\,,\,\lambda_1g_1(x)=-\lambda_1 x_1^2=0\,.
\end{equation}
It is equivalent to $x_1=0$. 
Thus $C(g)=\{(0,t)\,:\,t\in\R\}$ is infinite, which implies that $f(C(g))=\{t^2\,:\,t\in\R\}=[0,\infty)$  is infinite.
Let $h_\text{FJ}$ be as in \eqref{eq:.polyFJ}.
Then we obtain
\begin{equation}
h_\text{FJ}=(\lambda_0\begin{bmatrix}
1\\
2x_2
\end{bmatrix}-\lambda_1\begin{bmatrix}
-2x_1\\
0
\end{bmatrix},-\lambda_1 x_1^2,1-\lambda_0^2-\lambda_1^2)\,.
\end{equation}
Let $(x,\bar\lambda)\in V(h_\text{FJ})$. 
We get $-\lambda_1x_1^2=0$, so either $\lambda_1=0$ or $x_1=0$.
If $\lambda_1=0$, we get $\lambda_0\in\{\pm 1\}$ (since $1=\lambda_0^2+\lambda_1^2$) and $\lambda_0=0$. 
Hence this is impossible.
Consequently, $x_1=0$ implies $\lambda_0=0$, which gives $\lambda_1\in\{\pm 1\}$. 
It implies that $V(h_\text{FJ})=\{(0,t,0,\pm 1)\,:\,t\in\R\}$.
Since $f(C(g))$ is infinite, the assumption of Theorem \ref{theo:rep} does not hold in this case.
However, it is not hart to check that $q=x_2^2$ satisfies $q\in  P(g)[x,\bar \lambda]$ and $f-q$ vanishes on $V(h_\text{FJ})$.
It is interesting to know clearly the extension of Theorem \ref{theo:rep} to the case where $f(C(g))$ is infinite.
\end{example}
In the last three examples, the sets of global minimizers are finite. 
The following example shows the case where the set of global minimizers is infinite:
\begin{example}\label{exam:inf.minimizer}
Consider the problem \eqref{eq:pop} with $n=2$, $m=1$, $f=x_1-x_2$ and $g=(g_1)=(-(x_1-x_2)^2)$. 
Then the set of global minimizers $\{(t,t)\,:\,t\in\R\}$ is infinite.
The condition $\rank(\varphi^g(x))< m$ can be expressed as
\begin{equation}
\exists \lambda_1\in \R\backslash\{0\}\,:\,\lambda_1\nabla g_1(x)=\lambda_1\begin{bmatrix}
-2(x_1-x_2)\\
-2(x_2-x_1)
\end{bmatrix}=0\,,\,\lambda_1g_1(x)=-\lambda_1 (x_1-x_2)^2=0\,.
\end{equation}
It is equivalent to $x_2=x_1$. 
Thus $C(g)=\{(t,t)\,:\,t\in\R\}$ is infinite.
However, $f(C(g))=\{0\}$ is singleton.
Thus Theorem \ref{theo:rep} yields that there exists $q\in P(g)[x,\bar \lambda]$ with $\bar\lambda=(\lambda_0,\lambda_1)$ such that $f-q$ vanishes on $V(h_\text{FJ})$ with $h_\text{FJ}$ defined as in \eqref{eq:.polyFJ}, i.e.,
\begin{equation}
h_\text{FJ}=(\lambda_0\begin{bmatrix}
1\\
-1
\end{bmatrix}-\lambda_1\begin{bmatrix}
-2(x_1-x_2)\\
-2(x_2-x_1)
\end{bmatrix},-\lambda_1 (x_1-x_2)^2,1-\lambda_0^2-\lambda_1^2)\,.
\end{equation}
It is easy to check that  $V(h_\text{FJ})=\{(t,t,0,\pm 1)\,:\,t\in\R\}$.
It yields the selection $q=0$ in this case.
\end{example}

We show in the following three examples of problem \eqref{eq:pop} that the set of critical points $C(g)$ is empty, and so is its image under $f$:

\begin{example}\label{exam:ball}
Consider the problem \eqref{eq:pop} with $m=1$ and $g=(1-x_1^2-\dots-x_n^2)$. 
Then $S(g)$ is the unit ball.
The condition $\rank(\varphi^g(x))< m$ can be expressed as
\begin{equation}
\exists \lambda_1\in \R\backslash\{0\}\,:\,\lambda_1\nabla g_1(x)=-2\lambda_1x=0\,,\,\lambda_1g_1(x)=\lambda_1 (1-x_1^2-\dots-x_n^2)=0\,.
\end{equation}
It is equivalent to $x=0$ and $1-x_1^2-\dots-x_n^2=0$, and hence it is impossible. 
Thus we get $C(g)=\emptyset$ which implies  $f(C(g))=\emptyset$.
Note that $S(g)$ satisfies the Archimedean condition in this case.
From this, Theorem \ref{theo:rep} yields that if $f$ is non-negative on $S(g)$, there exists $q\in Q(g)[x,\bar \lambda]$ with $\bar\lambda=(\lambda_0,\lambda_1)$ such that $f-q$ vanishes on $V(h_\text{FJ})$ with $h_\text{FJ}$ defined as in \eqref{eq:.polyFJ}, i.e.,
\begin{equation}
h_\text{FJ}=(\lambda_0\nabla f-2\lambda_1x,\lambda_1(1-x_1^2-\dots-x_n^2),1-\lambda_0^2-\lambda_1^2)\,.
\end{equation}
\end{example}
Let $e_1,\dots,e_n$ be the canonical basis of $\R^n$.
\begin{example}\label{exam:box}
Consider the problem \eqref{eq:pop} with $m=n$ and $g=(1-x_1^2,\dots,1-x_n^2)$.
Then $S(g)$ is the unit hypercube.
The condition $\rank(\varphi^g(x))< m$ can be expressed as
\begin{equation}
\exists \lambda\in \R^n\backslash\{0\}\,:\,\sum_{j=1}^n\lambda_j\nabla g_j(x)=-2\sum_{j=1}^n\lambda_jx_je_j=0\,,\,\lambda_jg_j(x)=\lambda_j (1-x_j^2)=0\,.
\end{equation}
It implies that $\lambda_j=\lambda_j (1-x_j^2)-\frac{1}{2}x_j(-2\lambda_jx_j)=0$, $j=1,\dots,m$ which contradicts $\lambda\ne 0$. 
Thus we get $C(g)=\emptyset$, which implies $f(C(g))=\emptyset$.
Note that $S(g)$ satisfies the Archimedean in this case.
From this, Theorem \ref{theo:rep} yields that if $f$ is non-negative on $S(g)$, there exists $q\in Q(g)[x,\bar \lambda]$ with $\bar\lambda=(\lambda_0,\lambda_1)$ such that $f-q$ vanishes on $V(h_\text{FJ})$ with $h_\text{FJ}$ defined as in \eqref{eq:.polyFJ}, i.e.,
\begin{equation}
h_\text{FJ}=(\lambda_0\nabla f-2\sum_{j=1}^n\lambda_jx_je_j,\lambda_1(1-x_1^2),\dots,\lambda_n(1-x_n^2),1-\lambda_0^2-\dots-\lambda_n^2)\,.
\end{equation}
\end{example}

\begin{example}\label{exam:simplex}
Consider the problem \eqref{eq:pop} with $m=n+1$ and $g=(x_1,\dots,x_n,1-x_1-\dots-x_n)$.
Then $S(g)$ is the unit simplex.
The condition $\rank(\varphi^g(x))< m$ can be expressed as
\begin{equation}
\exists \lambda\in \R^{n+1}\backslash\{0\}\,:\,\begin{cases}
\sum_{j=1}^n\lambda_j\nabla g_j(x)+\lambda_{n+1}\nabla g_{n+1}(x)=\sum_{j=1}^n(\lambda_j-\lambda_{n+1})e_j=0\,,\\\lambda_jg_j(x)=\lambda_j x_j=0\,,\,j=1,\dots,n\,,\\
\lambda_{n+1}g_{n+1}(x)=\lambda_{n+1}(1-x_1-\dots-x_n)=0\,.
\end{cases}
\end{equation}
It implies that $\lambda_j=\lambda_{n+1}$, $j=1,\dots,n$, and
$\lambda_{n+1}=\lambda_{n+1}\sum_{j=1}^nx_j=\sum_{j=1}^n\lambda_jx_j=0$.
Thus we get $\lambda=0$ which contradicts $\lambda\ne 0$. 
Thus we obtain $C(g)=\emptyset$ which implies $f(C(g))=\emptyset$.
Note that $S(g)$ satisfies the Archimedean in this case, as shown in \cite{jacobi2001distinguished}.
From this, Theorem \ref{theo:rep} shows that if $f$ is non-negative on $S(g)$, there exists $q\in Q(g)[x,\bar \lambda]$ with $\bar\lambda=(\lambda_0,\lambda_1)$ such that $f-q$ vanishes on $V(h_\text{FJ})$ with $h_\text{FJ}$ defined as in \eqref{eq:.polyFJ}, i.e.,
\begin{equation}
h_\text{FJ}=(\lambda_0\nabla f-\sum_{j=1}^n(\lambda_j-\lambda_{n+1})e_j,\lambda_1x_1,\dots,\lambda_nx_n,\lambda_{n+1}(1-\sum_{j=1}^nx_j),1-\sum_{j=0}^{n+1}\lambda_j^2)\,.
\end{equation}
\end{example}

Next, we consider an example of problem \eqref{eq:pop} whose optimal value $f^\star$ is not attained:
\begin{example}\label{exam:not.attain}
Consider the problem \eqref{eq:pop} with $n=2$, $m=1$, $f=x_1$ and $g=(g_1)=(x_1x_2^2-1)$. 
It is not hard to prove that $f^\star=0$.
Note that problem \eqref{eq:pop} does not have any global minimizer.
The condition $\rank(\varphi^g(x))< m$ can be expressed as
\begin{equation}
\exists \lambda_1\in \R\backslash\{0\}\,:\,\lambda_1\nabla g_1(x)=\lambda_1\begin{bmatrix}
x_2^2\\
2x_1x_2
\end{bmatrix}=0\,,\,\lambda_1g_1(x)=\lambda_1 (x_1x_2^2-1)=0\,.
\end{equation}
It implies that $\lambda_1x_2^2=0$ which gives $\lambda_1=\lambda_1x_1x_2^2=0$.
This contradicts $\lambda_1\ne 0$. 
Thus we get $C(g)=\emptyset$, in consequence, $f(C(g))=\emptyset$.
From this, Theorem \ref{theo:rep} yields that there exists $q\in P(g)[x,\bar \lambda]$ with $\bar\lambda=(\lambda_0,\lambda_1)$ such that $f-q$ vanishes on $V(h_\text{FJ})$ with $h_\text{FJ}$ defined as in \eqref{eq:.polyFJ}, namely
\begin{equation}
h_\text{FJ}=(\lambda_0\begin{bmatrix}
1\\
0
\end{bmatrix}-\lambda_1\begin{bmatrix}
x_2^2\\
2x_1x_2
\end{bmatrix},\lambda_1 (x_1x_2^2-1),1-\lambda_0^2-\lambda_1^2)\,.
\end{equation}
Let $(x,\bar\lambda)\in V(h_\text{FJ})$. 
We get $\lambda_1x_1x_2=0$ which gives $\lambda_1=\lambda_1x_1x_2^2=0$.
It follows that $\lambda_0=0$ which is impossible since $1=\lambda_1^2+\lambda_2^2=0$.
Thus $V(h_\text{FJ})=\emptyset$ yields the selection $q=0$.
\end{example}

We indicate the following counterexample, and consequently, the assumption that $f(C(g))$ is finite in Theorem \ref{theo:rep} cannot be removed:
\begin{example}\label{exam:counterexample}
Consider the problem \eqref{eq:pop} with $n=2$, $m=2$, $f=x_1+x_2$ and $g=(-x_1^2,-x_2^2)$. 
Then we get $f^\star=0$ and $(0,0)$ is the unique global minimizer for this problem.
It is not hard to prove the Karush--Kuhn--Tucker conditions do not hold for this problem at any global minimizer.
The condition $\rank(\varphi^g(x))< m$ can be expressed as
\begin{equation}\label{eq:rank.deficient}
\exists \lambda\in \R^2\backslash\{0\}\,:\,\begin{cases}
\lambda_1\nabla g_1(x)+\lambda_2\nabla g_2(x)=\lambda_1\begin{bmatrix}
-2x_1\\
0
\end{bmatrix}+\lambda_2\begin{bmatrix}
0\\
-2x_2
\end{bmatrix}=0\,,\\
\lambda_jg_j(x)=-\lambda_j x_j^2=0\,,\,j=1,2\,.
\end{cases}
\end{equation}
Since $\lambda\ne 0$, either $\lambda_1$ or $\lambda_2$ is non-zero.
If $\lambda_1\ne 0$, we get $x_1=0$.
If $\lambda_2\ne 0$, we get $x_2=0$.
It implies that $x_1=0$ or $x_2=0$.
Thus we get $C(g)=\{(t,0),(0,t)\,:\,t\in\R\}$ which gives that $f(C(g))=\R$  is infinite.
Let $h_\text{FJ}$ be as in \eqref{eq:.polyFJ}.
Then we obtain
\begin{equation}
h_\text{FJ}=(\lambda_0\begin{bmatrix}
1\\
1
\end{bmatrix}-\lambda_1\begin{bmatrix}
-2x_1\\
0
\end{bmatrix}-\lambda_2\begin{bmatrix}
0\\
-2x_2
\end{bmatrix},-\lambda_1 x_1^2,-\lambda_2 x_2^2,1-\lambda_0^2-\lambda_1^2-\lambda_2^2)\,.
\end{equation}
Let $(x,\bar\lambda)\in V(h_\text{FJ})$. 
We get $\lambda_1x_1=\lambda_2x_2=0$ which gives $\lambda_0=0$.
Since $\lambda_1x_1=0$, either $\lambda_1=0$ or $x_1=0$ holds.
If $\lambda_1=0$, we get $\lambda_2\in\{\pm 1\}$ which gives $x_2=0$.
If $x_1=0$, we consider $\lambda_2x_2=0$ which gives $\lambda_2=0$ or $x_2=0$.
If $x_1=0$ and $\lambda_2=0$, it implies that $\lambda_1\in\{\pm 1\}$.
If $x_1=0$ and $x_2=0$, it follows that $\lambda_1^2+\lambda_2^2=1$.
Thus we see that
\begin{equation}
V(h_\text{FJ})=\{(t,0,0,0,\pm 1),(0,t,0,\pm 1,0),(0,0,0,\cos{t},\sin{t})\,:\,t\in\R\}\,.
\end{equation}
Since $f(C(g))$ is infinite, the assumption of Theorem \ref{theo:rep} does not hold in this case.
Assume that there exists $q\in  P(g)[x,\bar \lambda]$ such that $f-q$ vanishes on $V(h_\text{FJ})$.
We write $q=\sigma_0+\sigma_1g_1+\sigma_2g_2+\sigma_3g_1g_2$ for some $\sigma_j\in\Sigma^2[x]$.
For all $t\in\R$, since $(t,0,0,0,1)\in V(h_\text{FJ})$, we have $f(t,0)=q(t,0,0,0,1)$. 
Setting $\psi_j=\sigma_j(t,0,0,0,1)\in\Sigma^2[t]$, we obtain $t=\psi_0-\psi_1t^2$. 
It implies that $\psi_0=t(1+t\psi_1)$ which leads to $\psi_0=t^2\xi$ for some $\xi\in\Sigma^2[t]$ since $\psi_0\in \Sigma^2[t]$.
We now get $t=t^2(\xi-\psi_1)$ which is impossible.
Hence there does not exist $q\in  P(g)[x,\bar \lambda]$ such that $f-q$ vanishes on $V(h_\text{FJ})$.
\end{example}

\section{Exact polynomial optimization}
\label{sec:application}
\subsection{Moment-SOS relaxations}
In this subsection we recall some preliminaries of the Moment-SOS relaxations originally developed by Lasserre in \cite{lasserre2001global}.
Given $d\in\N$, let $\N^n_d:=\{\alpha\in\N^n\,:\,\sum_{j=1}^n \alpha_j\le d\}$.
Given $d\in\N$, we denote by $v_d$ the vector of monomials in $x$ of degree at most $d$, i.e., $v_d=(x^\alpha)_{\alpha\in\N^n_d}$ with $x^\alpha:=x_1^{\alpha_1}\dots x_n^{\alpha_n}$.
For each $p\in\R[x]_d$, we write $p=c(p)^\top v_d=\sum_{\alpha\in\N^n_d}p_\alpha x^\alpha$, where $c(p)$ is denoted by the vector of coefficient of $p$, i.e., $c(p)=(p_\alpha)_{\alpha\in\N^n_d}$ with $p_\alpha\in\R$.
Given $A\in\R^{r\times r}$ being symmetric, we say that $A$ is positive semidefinite, denoted by $A\succeq 0$, if every eigenvalue of $A$ is non-negative.

The following lemma shows the connection between sums of squares and semidefinite programs (see, e.g., \cite{boyd2004convex}):
\begin{lemma}\label{lem:sdp.sos}
Let $\sigma\in\R[x]$ and $d\in\N$ such that $2d\ge\deg(\sigma)$.
Then $\sigma\in\Sigma^2[x]$ iff there exists $G\succeq 0$ such that $\sigma=v_d^\top Gv_d$.
\end{lemma}

Given $y=(y_\alpha)_{\alpha\in\N^n}\subset \R$, let $L_y:\R[x]\to\R$ be the Riesz linear functional defined by $L_y(p)=\sum_{\alpha\in\N^n} p_\alpha y_\alpha$ for every $p\in\R[x]$.
Given $d\in\N$, $p\in\R[x]$ and $y=(y_\alpha)_{\alpha\in\N^n}\subset \R$, let $M_d(y)$ be the moment matrix of order $d$ defined by $(y_{\alpha+\beta})_{\alpha,\beta\in\N^n_d}$ and let $M_d(py)$ be the localizing matrix of order $d$ associated with $p$ defined by $(\sum_{\gamma\in\N^n}p_\gamma y_{\alpha+\beta+\gamma})_{\alpha,\beta\in\N^n_d}$.

Given $g_1,\dots,g_m\in\R[x]$, let $Q_d(g)[x]$ be the truncated quadratic module of order $d$ associated with $g=(g_1,\dots,g_m)$ defined by
\begin{equation}
Q_d(g)[x]=\{\sigma_0+\sum_{j=1}^m\sigma_jg_j\,:\,\sigma_j\in\Sigma^2[x]\,,\,\deg(\sigma_0)\le 2d\,,\,\deg(\sigma_jg_j)\le 2d\}\,.
\end{equation}
Given $h_1,\dots,h_l\in\R[x]$, let $I_d(h)$ be the truncated ideal of order $d$ associated with $h=(h_1,\dots,h_l)$ defined by
\begin{equation}
I_d(h)[x]=\{\sum_{j=1}^l\psi_jh_j\,:\,\psi_j\in\R[x]\,,\,\deg(\psi_jh_j)\le 2d\}\,.
\end{equation}

Given $k\in\N$ and $f,g_1,\dots,g_m,h_1,\dots,h_l\in\R[x]$, consider the following primal-dual semidefinite programs associated with $f$, $g=(g_1,\dots,g_m)$ and $h=(h_1,\dots,h_l)$:
\begin{equation}\label{eq:mom.relax}
\begin{array}{rl}
\tau_k(f,g,h):=\inf\limits_y& L_y(f)\\
\text{s.t} &M_k(y)\succeq 0\,,\,M_{k-d_j}(g_jy)\succeq 0\,,\,j=1,\dots,m\,,\\
&M_{k-r_t}(h_ty)=0\,,\,t=1,\dots,l\,,\,y_0=1\,,
\end{array}
\end{equation}
\begin{equation}\label{eq:sos.relax}
\begin{array}{rl}
\rho_k(f,g,h):=\sup\limits_{\xi,G_j,u_t} & \xi\\
\text{s.t} & G_j\succeq 0\,,\\
&f-\xi=v_k^\top G_0v_k+\sum_{j=1}^m g_jv_{k-d_j}^\top G_jv_{k-d_j}\\
&\qquad\qquad+\sum_{t=1}^l h_tu_t^\top v_{2r_t}\,,\\
\end{array}
\end{equation}
where $d_j=\lceil \deg(g_j)/2\rceil$ and $r_t=\lceil \deg(h_t)/2\rceil$.
Using Lemma \ref{lem:sdp.sos}, we obtain 
\begin{equation}\label{eq:equi.sos}
\rho_k(f,g,h):=\sup_{\xi\in\R}\{ \xi\,:\,f-\xi\in Q_k(g)[x]+I_k(h)[x]\}\,.
\end{equation}
Primal-dual semidefinite programs \eqref{eq:mom.relax}-\eqref{eq:sos.relax} are known as the Moment-SOS relaxations of order $k$ for problem
\begin{equation}\label{eq:pop.equality}
\bar f^\star:=\inf\limits_{x\in S(g)\cap V(h)} f(x)\,.
\end{equation}
We state in the following lemma some recent results involving the Moment-SOS relaxations:
\begin{lemma}\label{lem:mom.sos}
Let $f,g_1,\dots,g_m,h_1,\dots,h_l\in\R[x]$. 
Let $\bar f^\star$ be as in \eqref{eq:pop.equality} with $g=(g_1,\dots,g_m)$ and $h=(h_1,\dots,h_l)$. 
Then the following statements hold:
\begin{enumerate}
\item For every $k\in\N$, $\tau_k(f,g,h)\le \tau_{k+1}(f,g,h)$ and $\rho_k(f,g,h)\le \rho_{k+1}(f,g,h)$.
\item For every $k\in\N$, $\rho_k(f,g,h)\le \tau_{k}(f,g,h)\le \bar f^\star$.
\item If $S(g)\cap V(h)$ has a non-empty interior, for $k\in\N$ sufficiently large, the Slater condition holds for the Moment relaxation \eqref{eq:mom.relax} of order $k$.
\item If $S(g)\cap V(h)$ satisfies the Archimedean condition, $\rho_k(f,g,h)\to \bar f^\star$ as $k\to \infty$.
\item If there exists $R>0$ such that $g_m+h_l=R-x_1^2-\dots-x_n^2$, for $k\in\N$ sufficiently large, the Slater condition holds for the SOS relaxation \eqref{eq:sos.relax} of order $k$.
\item If there exists $q\in Q(g)[x]$ such that $f-\bar f^\star-q$ vanishes on $V(h)$, then there exists $k\in\N$ such that $\rho_k(f,g,h)=\bar f^\star$.
\end{enumerate}
\end{lemma}
\begin{proof}
The first four statements are proved by Lasserre in \cite{lasserre2001global}.
The proof of the fifth statement can be found in \cite{mai2020exploiting}.
The final statement is based on Nie's technique in \cite{nie2014optimality}, which is sketched as follows: 
Let $u=f-\bar f^\star-q$.
By assumption,we get  $u=0$ on $V(h)$. 
From this, Krivine--Stengle's Nichtnegativstellens\"atze \cite{krivine1964anneaux} yields that there exists a positive integer $r$ and $\sigma \in \Sigma^2[x]$ such that $u^{2r} + \sigma \in I(h)[x]$.
Let $c=\frac{1}{2r}$. 
Then it holds that $1+t+ct^{2r}\in\Sigma^2[t]$.
Thus for all $\varepsilon>0$, we have
\begin{equation}
f-\bar f^\star+\varepsilon=q + \varepsilon(1+\frac{u}{\varepsilon}+c\left(\frac{u}{\varepsilon}\right)^{2r})-c\varepsilon^{1-2r}(u^{2r} + \sigma ) +c\varepsilon^{1-2r}\sigma\in Q(g)[x]+I(h)[x]\,.
\end{equation}
Moreover, the degree of the right-hand side has an upper bound independent from $\varepsilon$.
This implies that there exists $k\in\N$ such that for all $\varepsilon>0$, $f-\bar f^\star+\varepsilon \in Q_k(g)[x]+I_k(h)[x]$.
Then we for all $\varepsilon>0$, $\bar f^\star-\varepsilon$ is a feasible solution of \eqref{eq:equi.sos} of the value $\rho_k(f,g,h)$.
It gives $\rho_k(f,g,h)\ge \bar f^\star-\varepsilon$, for all $\varepsilon>0$, and, in consequence, we get $\rho_k(f,g,h)\ge \bar f^\star$.
Using the second statement, we obtain that $\rho_k(f,g,h)= \bar f^\star$, yielding the final statement.
\end{proof}
\begin{remark}
In the final statement of Lemma \ref{lem:mom.sos}, if we assume further that $I(h)$ is real radical, and $\rho_k(f,g,h)$ is attained.
However, there is a case where $\rho_k(f,g,h)$ is not attained, although $\rho_k(f,g,h)= \bar f^\star$.
Example \ref{exam:1} is an instance for this (see at the end of \cite[Section 3]{nie2014optimality}).
Fortunately, if there exists $R>0$ such that $g_m+h_l=R-x_1^2-\dots-x_n^2$, for $k\in\N$ sufficiently large, the fifth statement of Theorem \ref{lem:mom.sos} shows that the Slater condition holds for the SOS relaxation \eqref{eq:sos.relax} of order $k$, and hence the Moment relaxation \eqref{eq:mom.relax} has at least one global minimizer.
\end{remark}
As a consequence of Lemma \ref{lem:FJ}, the following lemma is obtained:
\begin{lemma}\label{lem:equi.prob}
Let $f,g_1,\dots,g_m\in\R[x]$. 
Let $f^\star$ be as in problem \eqref{eq:pop} with $g=(g_1,\dots,g_m)$.
Let $h_\text{FJ}$ be as in \eqref{eq:.polyFJ}.
If problem \eqref{eq:pop} has a global minimizer, then it holds that
\begin{equation}\label{eq:eqi.POP}
\begin{array}{rl}
f^\star=\min\limits_{x,\bar\lambda}& f(x)\\
\text{s.t.}& x\in S(g)\,,\,(x,\bar\lambda)\in V(h_\text{FJ})\,.
\end{array}
\end{equation}
\end{lemma}
We present  in the following theorem the main application of Theorem \ref{theo:rep} to polynomial optimization:
\begin{theorem}\label{theo:pop}
Let $f,g_1,\dots,g_m\in\R[x]$. 
Let $f^\star$ be as in problem \eqref{eq:pop} with $g=(g_1,\dots,g_m)$.
Let $h_\text{FJ}$ be as in \eqref{eq:.polyFJ}.
Assume that problem \eqref{eq:pop} has at least one global minimizer, and $f(C(g))$ is finite.
Then there exists $k\in\N$ such that $\rho_k(f,\Pi g,h_\text{FJ})=f^\star$, where $\Pi g$ is defined as in \eqref{eq:prod.g}.
Moreover, if $S(g)$ satisfies the Archimedean condition, there exists $k\in\N$ such that $\rho_k(f,g,h_\text{FJ})=f^\star$.
Furthermore, if these exists $R>0$ such that $g_m=R-x_1^2-\dots-x_n^2$, for $k\in\N$ sufficiently large, the Slater condition holds for the SOS relaxation \eqref{eq:sos.relax} of order $k$ with $h=h_\text{FJ}$.
\end{theorem}
\begin{proof}
We first prove $\rho_k(f,\Pi g,h_\text{FJ})=f^\star$ for some $k\in\N$.
Since $S(g)=S(\Pi g)$, Lemma \ref{lem:equi.prob} implies that
\begin{equation}
\begin{array}{rl}
f^\star=\min\limits_{x,\bar\lambda}& f(x)\\
\text{s.t.}& x\in S(\Pi g)\,,\,(x,\bar\lambda)\in V(h_\text{FJ})\,,
\end{array}
\end{equation}
By assumption, Theorem \ref{theo:rep} yields that there exists $q\in P(g)[x,\bar \lambda]=Q(\Pi g)[x,\bar \lambda]$ such that $f-f^\star-q$ vanishes on $V(h_\text{FJ})$.
Applying the final statement of Lemma \ref{lem:mom.sos} (by replacing $g$ with $\Pi g$), we obtain the first statement.
Assume that $S(g)$ satisfies the Archimedean condition.
The proof of the equality $\rho_k(f,g,h_\text{FJ})=f^\star$ for some $k\in\N$ is similar.
Now assume that there exists $R>0$ such that $g_m=R-x_1^2-\dots-x_n^2$.
Let us prove the final statement.
By definition of $h_{FJ}$, the final entry of  $h_{FJ}$ is $b=1-\lambda_0^2-\dots-\lambda_m^2$ which implies that
\begin{equation}
g_m+b=(R+1)-x_1^2-\dots-x_n^2- \lambda_0^2-\dots- \lambda_{m}^2\,.
\end{equation}
From this, the fifth statement of Lemma \ref{lem:mom.sos} shows that for $k\in\N$ sufficiently large, the Slater condition holds for the SOS relaxation \eqref{eq:sos.relax} of order $k$ with $h=h_\text{FJ}$, yielding the final statement.
\end{proof}
Combining Theorem \ref{theo:pop}, Examples \ref{exam:ball}, \ref{exam:box}, and \ref{exam:simplex}, we obtain the following corollary:
\begin{corollary}\label{coro:ball.box.simplex}
Let $f,g_1,\dots,g_m\in\R[x]$.
Let $f^\star$ be as in problem \eqref{eq:pop} with $g=(g_1,\dots,g_m)$.
Let $h_\text{FJ}$ be as in \eqref{eq:.polyFJ}.
Assume that problem \eqref{eq:pop} has at least one global minimizer, and one of the following conditions holds:
\begin{enumerate}
    \item $g=(1-x_1^2-\dots-x_n^2)$;
    \item $g=(1-x_n^2,\dots,1-x_n^2)$;
    \item $g=(x_1,\dots,x_n,1-x_1-\dots-x_n)$.
\end{enumerate}
Then there exists $k\in\N$ such that $\rho_k(f,g,h_\text{FJ})=f^\star$.
\end{corollary}
The following two examples are given in \cite[Example 3.3]{nie2014optimality}:
\begin{example}
Let $\varepsilon>0$.
Consider the problem \eqref{eq:pop} with $n=3$, $m=1$, $f=x_1^4x_2^2+x_1^2x_2^4+x_3^6-3x_1^2x_2^2x_3^2+\varepsilon(x_1^2+x_2^2+x_3^2)$ and $g=(1-x_1^2-x_2^2-x_3^2)$. 
For $\varepsilon>0$ sufficiently small, Lasserre's hierarchy for this problem does not have finite convergence, as shown by Marshall
\cite[Example 2.4]{marshall2006representations}.
However, Corollary \ref{coro:ball.box.simplex} yields that there exists $k\in\N$ such that $\rho_k(f,g,h_\text{FJ})=f^\star$.
\end{example}
\begin{example}
Consider the problem \eqref{eq:pop} with $n=2$, $m=3$, $f=x_1x_2+x_1^3+x_2^3$ and $g=(x_1,x_2,1-x_1-x_2)$. 
As shown by Scheiderer \cite[Remark 3.9]{scheiderer2005distinguished}, Lasserre's hierarchy for this problem does not have finite convergence.
However, Corollary \ref{coro:ball.box.simplex} yields that there exists $k\in\N$ such that $\rho_k(f,g,h_\text{FJ})=f^\star$.
\end{example}

The following two examples are given in \cite[Example 5.6]{nie2013exact}:
\begin{example}
Consider the problem \eqref{eq:pop} with $n=m=3$, $f=x_1^4x_2^2+x_2^4x_3^2+x_3^4x_1^2-3x_1^2x_2^2x_3^2$ and $g=(1-x_1^2,1-x_2^2,1-x_3^2)$. 
As shown by Nie \cite[Example 5.6]{nie2013exact}, Lasserre's hierarchy for this problem does not have finite convergence.
However, Corollary \ref{coro:ball.box.simplex} yields that there exists $k\in\N$ such that $\rho_k(f,g,h_\text{FJ})=f^\star$.
\end{example}

\subsection{Numerical experiments}
\label{sec:num}
In this subsection, we report numerical results produced by the SOS relaxations of the values $\rho_k(f,g,0)$ and $\rho_k(f,g,h_\text{FJ})$ for problem \eqref{eq:pop}, where $h_\text{FJ}$ is defined as in \eqref{eq:.polyFJ}. 
The former is the standard semidefinite program in \cite{lasserre2001global} while the latter is the semidefinite program modeled by our method in this paper.
We indicate the data of each semidefinite program, namely ``value", ``time", and ``size" correspond to the numerical value of the optimal value of the semidefinite program, the running time in seconds to obtain this numerical value, and the size of the semidefinite program, respectively.
Here the size of the semidefinite program includes the largest matrix size, the number of affine constraints, the number of scalar variables, and the number of matrix variables.

The experiments are performed in Julia 1.3.1 with softwares TSSOS \cite{wang2021tssos} and Mosek 9.1 \cite{mosek2010mosek}.
We use a desktop computer with an Intel(R) Core(TM) i7-8665U CPU @ 1.9GHz $\times$ 8 and 31.2 GB of RAM. 

Our test problem is taken from Example \ref{exam:2}, namely $n=2$, $m=1$, $f=(x_1+1)^2+x_2^2-1$ and $g=(g_1)=(x_1^3-x_2^2)$.
It is clear that $f^\star=0$ which is attained at the unique global minimizer $(0,0)$ for problem \eqref{eq:pop}.
Moreover, the Karush--Kuhn--Tucker conditions do not hold at this minimizer.
By using \cite[Proposition 3.4]{nie2014optimality}, we get $\rho_k(f,g,0)< f^\star$ for all $k\in\N$.
As shown in Example \ref{exam:2}, $f(C(g))$ is singleton.
From this, Theorem \ref{theo:pop} yields that  $\rho_d(f,g,h_\text{FJ})=f^\star$ for some $d\in\N$.
Thus we obtain $\rho_k(f,g,0)<f^\star=\rho_d(f,g,h_\text{FJ})$, for all $k\in\N$.

We display the numerical results in Table \ref{tab:bench}.
\begin{table}
\footnotesize
\caption{\footnotesize Lower bounds on $f^\star=0$}
\label{tab:bench}
\begin{center}
\begin{tabular}{ cccc }
 \hline
\multicolumn{4}{c}{$\rho_k(f,g,h_\text{FJ})$} \\
\hline
$k$& value & time & size\\
$2$ &$-0.99984$ & 0.03& (15,70,34,1)\\
$3$ &$-0.48463$ & 0.08& (35,210,271,2)\\
$4$&$-0.00920$ & 0.33& (70,495,1501,2)\\
$5$ &$-0.00860$ & 1.63& (126,1001,6231,2)\\
$6$ &$-0.00831$ & 7.64& (210,1820,20973,2)\\
\hline
\multicolumn{4}{c}{$\rho_k(f,g,0)$}\\
\hline
$k$&value & time & size \\
$16$&$-0.04760$& 2.34 & (153,561,1,2)\\
$17$&$-0.04672$& 3.30 & (171,630,1,2)\\
$18$&$-0.04528$& 5.26 & (190,703,1,2)\\
$19$&$-0.04476$& 6.73 & (210,780,1,2)\\
$20$&$-0.04337$& 11.4 & (231,861,1,2)\\
 \hline
\multicolumn{4}{c}{$\rho_k(f,(g,b),0)$} \\
\hline
$k$&value & time & size \\
$16$&$-0.02370$& 5.31 & (153,561,1,3)\\
$17$&$-0.02306$& 6.97 & (171,630,1,3)\\
$18$&$-0.02284$& 7.91 & (190,703,1,3) \\
$19$&$-0.02198$& 11.9 & (210,780,1,3) \\
$20$&$-0.02148$& 24.1 & (231,861,1,3)\\
\end{tabular}
\end{center}
\end{table}
It shows that the numerical value $-0.00831$ of $\rho_6(f,g,h_\text{FJ})$ is the best lower bound on $f^\star$. 
It takes around $7$ seconds to obtain this numerical value. 
It is worth pointing out that the standard SOS relaxations of the values $\rho_k(f,g,0)$ and  $\rho_k(f,(g,b),0)$ cannot reach the bound $-0.00831$ in less than $10$ seconds despite using the additional ball constraint $b=1-x_1^2-x_2^2$.
Another observation is that the size of the SOS relaxations of the value $\rho_k(f,g,h_\text{FJ})$ grows faster than the ones of the values $\rho_k(f,g,0)$ and  $\rho_k(f,(g,b),0)$ when $k$ increases.
\section{Variations}
\label{sec:variation}
\subsection{Representations}
\label{sec:rep.variation}
We state in the following theorem the first variation of Theorem \ref{theo:rep}, where we assume that the image of the intersection of a semi-algebraic set with the set of its singularities is finite:
\begin{theorem}\label{theo:rep.variation}
Let $f,g_1,\dots,g_m\in\R[x]$. 
Assume that $f$ is non-negative on $S(g)$  with $g:=(g_1,\dots,g_m)$ and $f(C(g)\cap S(g))$ is finite. 
Then there exists $q\in P(g)[x,\bar \lambda]$ such that $f-q$ vanishes on $(S(g)\times\R^{m+1})\cap V(h_\text{FJ})$, where $\bar\lambda:=(\lambda_0,\dots,\lambda_m)$ and $h_\text{FJ}$ is defined as in \eqref{eq:.polyFJ}.
Moreover, if $S(g)$ satisfies the Archimedean condition, we can take $q\in Q(g)[x,\bar \lambda]$.
\end{theorem}
The proof of Theorem \ref{theo:rep.variation} is postponed to Section \ref{sec:proof.variation}.

Given a real matrix $A$, we denote by $\rank^+(A)$ the largest number of columns of $A$ whose convex hull over $\R$ has no zero.
Let  $C^+(g)$ be the set of critical points associated with $g$ defined by
\begin{equation*}
    C^+(g):=\{x\in \R^n\,:\,\rank^+(\varphi^g(x))< m\}.
\end{equation*}
We state the second variation of Theorem \ref{theo:rep} in the following theorem:
\begin{theorem}\label{theo:rep.plus}
Let $f,g_1,\dots,g_m\in\R[x]$. 
Assume that $f$ is non-negative on $S(g)$  with $g:=(g_1,\dots,g_m)$, and $f(C^+(g))$ is finite. 
Then there exists $q\in P(g)[x,\bar \lambda]$ such that $f-q$ vanishes on $V(h_\text{FJ}^+)$, where $\bar\lambda:=(\lambda_0,\dots,\lambda_m)$ and
\begin{equation}\label{eq:.polyFJ.plus}
    h_\text{FJ}^+:=(\lambda_0^2\nabla f-\sum_{j=1}^m \lambda_j^2 \nabla g_j,\lambda_1^2g_1,\dots,\lambda_m^2g_m,1-\sum_{j=0}^m\lambda_j^2)\,.
\end{equation}
Moreover, if $S(g)$ satisfies the Archimedean condition, we can take $q\in Q(g)[x,\bar \lambda]$.
\end{theorem}
\begin{proof}
To prove Theorem \ref{theo:rep.plus}, we do similarly to the proof of Theorem \ref{theo:rep} by replacing $C(g)$ and $h_\text{FJ}$ in Section \ref{sec:proof.rep} with $C^+(g)$ and $h_\text{FJ}^+$, respectively.
Note that the equality $C^+(g)=\pi(V(h_\text{FJ}^+)\cap \{\lambda_0=0\})$ follows from the following equivalences:
\begin{equation}\label{eq:equi.Cplus}
    \begin{array}{rl}
         &  x\in C^+(g)\\
        \Leftrightarrow & \rank^+(\varphi^g(x)) < m\\
        \Leftrightarrow & \exists \lambda\in\R^{m}\,:\, \sum_{j=1}^m\lambda_j^2=1\,,\, \sum_{j=1}^m \lambda_j^2 \nabla g_j( x)=0\,,\, \lambda_j^2 g_j(x) =0\\
        \Leftrightarrow & \exists \bar \lambda\in\R^{m+1}\,:\,
        \sum_{j=0}^m\lambda_j^2=1\,,\,\lambda_0=0\,,\,\lambda_0^2\nabla f(x)=\sum_{j=1}^m \lambda_j^2 \nabla g_j(x)\,,\, \lambda_j^2 g_j(x) =0\\
        \Leftrightarrow & \exists \bar \lambda\in\R^{m+1}\,:\,(x,\bar\lambda)\in V(h_\text{FJ}^+)\cap \{\lambda_0=0\}\\
        \Leftrightarrow& x\in\pi(V(h_\text{FJ}^+)\cap \{\lambda_0=0\})\,.
    \end{array}
\end{equation}
Moreover, the Lagrangian function becomes
$L(x,\bar \lambda) = f(x)+\sum_{j=1}^m \left(\frac{\lambda_j}{\lambda_0}\right)^2 g_j (x)$.
\end{proof}
We state the third variation of Theorem \ref{theo:rep} in the following theorem:
\begin{theorem}\label{theo:rep.plus2}
Let $f,g_1,\dots,g_m\in\R[x]$. 
Assume that $f$ is non-negative on $S(g)$  with $g:=(g_1,\dots,g_m)$ and $f(C^+(g)\cap S(g))$ is finite. 
Then there exists $q\in P(g)[x,\bar \lambda]$ such that $f-q$ vanishes on $S(g)\cap V(h_\text{FJ}^+)$, where $\bar\lambda:=(\lambda_0,\dots,\lambda_m)$ and $h_\text{FJ}^+$ is defined as in \eqref{eq:.polyFJ.plus}.
Moreover, if $S(g)$ satisfies the Archimedean condition, we can take $q\in Q(g)[x,\bar \lambda]$.
\end{theorem}
\begin{proof}
The proof is processed similarly to the one of Theorem \ref{theo:rep} by replacing $C(g)$ and $h_\text{FJ}$ in Section \ref{sec:proof.rep} with $C^+(g)$ and $h_\text{FJ}^+$, respectively.
\end{proof}

\begin{remark}
Since $C(g)\cap S(g)$, $C^+(g)$ and $C^+(g)\cap S(g)$ are subsets of $C(g)$, the assumptions of Theorems \ref{theo:rep.variation}, \ref{theo:rep.plus} and \ref{theo:rep.plus2} that $f(C(g)\cap S(g))$, $f(C^+(g))$ and $f(C^+(g)\cap S(g))$ are finite hold generically thanks to Theorem \ref{theo:gen}, respectively.
\end{remark}
\subsection{Proofs}
\label{sec:proof.variation}
We use the same notation as in Section \ref{sec:proof}.
We generalize Lemma \ref{lem:quadra} in the following lemma:
\begin{lemma}\label{lem:quadra.variation}
Let $f\in\R[x]$, let $W$ be a complex variety defined by finitely many polynomials in $\R[x]$, and let $A$ be a semi-algebraic subset of $\R^n$. 
Assume that $f(A\cap W)$ is finite. 
Then there exists a finite sequence of subsets $W_1,\dots,W_r$ such that the following conditions hold:
\begin{enumerate}
\item $W_1,\dots,W_r$ are pairwise disjoint complex varieties defined by finitely many polynomials in $\R[x]$;
\item for $j=1,\dots,r$, $W_j\subset W$, and  $f$ is constant on $W_j$;
\item $(W_1\cup\dots\cup W_r)\cap A=W\cap A$.
\end{enumerate}
\end{lemma}
The proof of Lemma \ref{lem:quadra.variation} is similar to the one of Lemma \ref{lem:quadra} since we only need to replace $\R^n$ in the proof of Lemma \ref{lem:quadra} with $A$.

The following lemma extends \cite[Lemma 3.3]{demmel2007representations} to the case of varieties defined by the Fritz John conditions:
\begin{lemma}\label{lem:constant.variation}
Let $f,g_1,\dots,g_m\in\R[x]$. 
Assume that $f(C(g))$ with $g:=(g_1,\dots,g_m)$ is finite.
Let $h_\text{FJ}$ be defined as in \eqref{eq:.polyFJ}.
Let $W$ be a semi-algebraically path connected component of $(S(g)\times\R^{m+1})\cap V(h_\text{FJ})$. Then $f$ is constant on $W$.
\end{lemma}
\begin{proof}
The proof is similar to the one of Lemma \ref{lem:constant}. 
When proving that $\tau\mapsto f(x(\tau))$ is constant on $[a_j,b_j]$, we note that \begin{equation}
f(S(g)\cap C(g))=f(S(g)\cap \pi(V(h_\text{FJ})\cap \{\lambda_0=0\}))
\end{equation}
is finite since $f(S(g)\cap\pi(V(h_\text{FJ})\cap \{\lambda_0=0\}))\supset f(x([\tau_1,\tau_2]))$, which is because
\begin{equation}
\begin{array}{rl}
S(g)\cap\pi(V(h_\text{FJ})\cap \{\lambda_0=0\})=&\pi((S(g)\times\R^{m+1})\cap V(h_\text{FJ})\cap \{\lambda_0=0\})\\
\supset& \pi(W\cap \{\lambda_0=0\}) \supset x([\tau_1,\tau_2])\,.
\end{array}
\end{equation}
\end{proof}

\subsubsection*{Proof of Theorem \ref{theo:rep}}
\begin{proof}
Using Lemma \ref{eq:component.semi}, we decompose $(S(g)\times\R^{m+1})\cap V(h_\text{FJ})$ into semi-algebraically path connected components:
$Z_1,\dots,Z_s$.
Lemma \ref{lem:constant.variation} yields that $f$ is constant on $Z_i$.
Thus $f((S(g)\times\R^{m+1})\cap V(h_\text{FJ}))$ is finite.
Observe that $(S(g)\times\R^{m+1})\cap V(h_\text{FJ})=(S(g)\times\R^{m+1})\cap V_\C(h_\text{FJ})$.
By using Lemma \ref{lem:quadra.variation}, we obtain a finite sequence of subsets $W_{1},\dots,W_{r}$ such that the following conditions hold:
\begin{itemize}
\item $W_{1},\dots,W_{r}$ are pairwise disjoint complex varieties defined by finitely many polynomials in $\R[x,\bar\lambda]$;
\item for $j=1,\dots,r$, $W_{j}\subset V_\C(h_\text{FJ})$, and  $f$ is constant on $W_{j}$;
\item $(W_{1}\cup\dots\cup W_{r})\cap (S(g)\times\R^{m+1})=V_\C(h_\text{FJ})\cap (S(g)\times\R^{m+1})=(S(g)\times\R^{m+1})\cap V(h_\text{FJ})$.
\end{itemize}
Let $D$ be the union of $W_{1},\dots,W_{r}$.
Let $b=1-\lambda_0^2-\dots-\lambda_m^2$.
From this, Lemma \ref{lem:const.func} yields that there exists $p\in P(g,b)[x,\bar\lambda]$ such that $f-p$ vanishes on $D\cap\R^{n+m+1}\supset (S(g)\times\R^{m+1})\cap V(h_\text{FJ})$.
We write $p$ as in \eqref{eq:preo.rep.b}
for some $\sigma_\alpha,\psi_\beta\in\Sigma^2[x,\bar{\lambda}]$.
Let 
$q=\sum_{\alpha\in\{0,1\}^m}\sigma_\alpha g^\alpha\in P(g)[x,\bar\lambda]$.
Since $b=0$ on $V(h_\text{FJ})$, it holds that $f=p=q$ on $(S(g)\times\R^{m+1})\cap V(h_\text{FJ})$.

Assume that $S(g)$ satisfies the Archimedean condition. Then there exists $R>0$ such that $g_{m+1}=R-x_1^2-\dots-x_n^2\in Q(g)[x]$.
It implies that $S(g,b)$ with $b=1-\lambda_0^2-\dots-\lambda_m^2$ satisfies the Archimedean condition due to \eqref{eq:large.ball}.
From this, Lemma \ref{lem:const.func} yields that there exists $p\in Q(g,b)[x,\bar\lambda]$ such that $f-p$ vanishes on $D\cap\R^{n+m+1}\supset (S(g)\times\R^{m+1})\cap V(h_\text{FJ})$.
We write $p$ as in \eqref{eq:quadra.rep.b}
for some $\sigma_j\in\Sigma^2[x,\bar{\lambda}]$.
Let 
$q=\sigma_0+\sum_{j=1}^m\sigma_j g_j\in Q(g)[x,\bar\lambda]$.
Since $b=0$ on $V(h_\text{FJ})$, $f=p=q$ on $(S(g)\times\R^{m+1})\cap V(h_\text{FJ})$.
This completes the proof. 
\end{proof}

\subsection{Examples}
\label{sec:examples.variation}
This section illustrates our Nichtnegativstellens\"atze stated in Theorems \ref{theo:rep.variation}, \ref{theo:rep.plus}, and \ref{theo:rep.plus2} with several explicit examples.
The following lemma shows a case where $S(g)$ is convex and $C^+(g)$ is empty:

\begin{lemma}\label{lem:convex.non-emptyinter}
Let $g=(g_1,\dots,g_m)$ with $g_j\in\R[x]$. 
Assume that each $g_j$ is concave, and $S(g)$ has a non-empty interior.
Then $S(g)$ is convex and $C^+(g)=\emptyset$.
\end{lemma}
\begin{proof}
It is a simple matter to prove that $S(g)$ is convex.
Assume by contradiction that there is $y\in C^+(g)$. 
Then by \eqref{eq:equi.Cplus}, there exists $\lambda\in\R^{m}\backslash\{0\}$ such that $\sum_{j=1}^m \lambda_j^2 \nabla g_j( y)=0$ and $\lambda_j^2 g_j(y) =0$.
Set $G(x)=\sum_{j=1}^m \lambda_j^2 g_j( x)$.
Then $G$ is concave since all $g_j$ is concave.
In addition, $\nabla G(y)=0$ yields that $G(y)=0$ is the maximal value of $G$.
Let $a$ be in the interior of $S(g)$. 
Then $0\ge G(a)=\sum_{j=1}^m \lambda_j^2 g_j(a)$ implies that $\lambda_1=\dots=\lambda_m=0$ since all $g_j(a)$ are positive.
This contradicts $\lambda\ne 0$, and hence it holds that $C^+(g)=\emptyset$.
\end{proof}
The following lemma is a consequence of Lemma \ref{lem:convex.non-emptyinter} and Theorem \ref{theo:rep.plus}:
\begin{lemma}\label{lem:rep.convex.nomempinter}
Let $f,g_1,\dots,g_m\in\R[x]$. 
Assume that $f$ is non-negative on $S(g)$  with $g:=(g_1,\dots,g_m)$, each $g_j$ is concave, and $S(g)$ has a non-empty interior.
Then $S(g)$ is convex, and there exists $q\in P(g)[x,\bar \lambda]$ such that $f-q$ vanishes on $V(h_\text{FJ}^+)$, where $\bar\lambda:=(\lambda_0,\dots,\lambda_m)$, and $h_\text{FJ}^+$ is defined as in \eqref{eq:.polyFJ.plus}.
Moreover, if $S(g)$ satisfies the Archimedean condition, we can take $q\in Q(g)[x,\bar \lambda]$.
\end{lemma}
To illustrate the representations in Lemma \ref{lem:rep.convex.nomempinter}, see Examples \ref{exam:ball}, \ref{exam:box} and \ref{exam:simplex}.
Note that if $C(g)=\emptyset$, then so is $C^+(g)$ since $C^+(g)\subset C(g)$.

Contrary to Lemma \ref{lem:rep.convex.nomempinter}, the following example shows the representation of any polynomial non-negative on a non-convex semi-algebraic set $S(g)$:
\begin{example}\label{exam:infinite.critical.variation}
Consider the problem \eqref{eq:pop} with $n=2$, $m=3$ and $g=(g_1,g_2,g_3)=(x_1+1,1-x_2^2,1-(x_1-1)^2-x_2^2)$. 
Then $S(g)$ is non-convex since it contains all points in the hypercube $[-1,1]^2$  but not in the open ball of center $(1,0)$ with unit radius.
The condition $\rank^+(\varphi^g(x))< m$ can be expressed as
\begin{equation}
\begin{cases}
\exists \lambda\in \R_+^3\backslash\{0\}\,:\,\sum_{j=1}^3\lambda_j\nabla g_j(x)=\lambda_1\begin{bmatrix}
1\\
0
\end{bmatrix}+\lambda_2\begin{bmatrix}
0\\
-2x_2
\end{bmatrix}+\lambda_3\begin{bmatrix}
2(x_1-1)\\
-2x_2
\end{bmatrix}=0\,,\\
\lambda_1g_1(x)=\lambda_1(x_1+1)=0\,,\,\lambda_2g_2(x)=\lambda_2(1-x_2^2)=0\,,\\
\lambda_3g_3(x)=\lambda_3(1-(x_1-1)^2-x_2^2)=0\,.
\end{cases}
\end{equation}
It implies that $\lambda_1=0$ or $x_1=-1$.
If $\lambda_1=0$, $2\lambda_3(x_1-1)=0$ which gives $\lambda_3=0$ or $x_1=1$.
If $\lambda_1=\lambda_3=0$, $\lambda_2\ne 0$ which implies $-2x_2=0$ and $1-x_2^2=0$, hence this is impossible.
If $\lambda_1=0$ and $x_1=1$, then $-2(\lambda_2+\lambda_3)x_2=0$ and $\lambda_2(1-x_2^2)=\lambda_3(1-x_2^2)=0$.
If $\lambda_1=0$, $x_1=1$ and $x_2=0$, then $\lambda_2=\lambda_3=0$ which contradicts $\lambda\ne 0$.
If $\lambda_1=0$, $x_1=1$ and $\lambda_2+\lambda_3=0$, then $\lambda_2=\lambda_3=0$ (since $\lambda_j\ge 0$) which also contradicts $\lambda\ne 0$.
Thus we get $x_1=-1$.
Then $\lambda_1-4\lambda_3=0$, $-2(\lambda_2+\lambda_3)x_2=0$, $\lambda_2(1-x_2^2)=0$ and $\lambda_3(-3-x_2^2)=0$.
It implies $\lambda_3=0$, so $\lambda_1=0$, $\lambda_2x_2=0$ and $\lambda_2(1-x_2^2)=0$.
Since $\lambda_1=\lambda_3=0$, we obtain $\lambda_2>0$ which gives $x_2=0=1-x_2^2$.
This is impossible, hence $C^+(g)=\emptyset$ which implies that $f(C^+(g))=\emptyset$. 
By Theorem \ref{theo:rep.plus}, if $f$ is non-negative on $S(g)$, there exists $q\in  P(g)[x,\bar \lambda]$ and $f-q$ vanishes on $V(h_\text{FJ}^+)$.  
\end{example}

In the following example, we reconsider Example \ref{exam:counterexample}, where Theorem \ref{theo:rep} is inapplicable, but Theorem \ref{theo:rep.variation} is applicable:
\begin{example}\label{exam:counterexample.variation}
Consider the problem \eqref{eq:pop} with $n=2$, $m=2$, $f=x_1+x_2$ and $g=(-x_1^2,-x_2^2)$. 
Then we get $f^\star=0$, $S(g)=\{(0,0)\}$ and $(0,0)$ is the unique global minimizer for this problem.
We get $C(g)=\{(t,0),(0,t)\,:\,t\in\R\}$ which gives that $f(C(g))=\R$  is infinite.
However $C(g)\cap S(g)=\{(0,0)\}$ implies $f(C(g)\cap S(g))=\{0\}$ is finite.
Let $h_\text{FJ}$ be as in \eqref{eq:.polyFJ}.
Then 
\begin{equation}
V(h_\text{FJ})=\{(t,0,0,0,\pm 1),(0,t,0,\pm 1,0),(0,0,0,\cos{t},\sin{t})\,:\,t\in\R\}\,.
\end{equation}
By Theorem \ref{theo:rep}, there exists $q\in  P(g)[x,\bar \lambda]$ such that $f-q$ vanishes on 
\begin{equation}
(S(g)\times\R^{m+1})\cap V(h_\text{FJ})=\{(0,0,0,\cos{t},\sin{t})\,:\,t\in\R\}\,.
\end{equation}
It is not hard to take $q=0$.
\end{example}

The following lemma shows a case where $C^+(g)$ is a singleton:

\begin{example}\label{exam:nonkkt}
Consider the problem \eqref{eq:pop} with $n=2$, $m=3$, $f=x_1-1$ and $g=(g_1,g_2,g_3)=(x_1,x_2,(x_1-1)^3-x_2)$. 
It is easy to check that $S(g)$ is non-convex and $f^\star=0$.
Moreover, $(1,0)$ is the unique global minimizer for this problem, and Karush--Kuhn--Tucker conditions do not hold for \eqref{eq:pop} at this point.
The condition $\rank^+(\varphi^g(x))< m$ can be expressed as
\begin{equation}
\begin{cases}
\exists \lambda\in \R_+^3\backslash\{0\}\,:\,\sum_{j=1}^3\lambda_j\nabla g_j(x)=\lambda_1\begin{bmatrix}
1\\
0
\end{bmatrix}+\lambda_2\begin{bmatrix}
0\\
1
\end{bmatrix}+\lambda_3\begin{bmatrix}
3(x_1-1)^2\\
-1
\end{bmatrix}=0\,,\\
\lambda_1g_1(x)=\lambda_1 x_1=0\,,\,\lambda_2g_2(x)=\lambda_2 x_2=0\,,\,\lambda_3g_3(x)=\lambda_3 ((x_1-1)^3-x_2)=0\,.
\end{cases}
\end{equation}
It implies that $\lambda_3=\lambda_2$ so $\lambda_1+3\lambda_2(x_1-1)^2=0$. 
It follows that $\lambda_1=0$ and $\lambda_2(x_1-1)=0$.
If $\lambda_2=0$, then $\lambda=0$, and hence this is impossible.
Thus we get $\lambda_2>0$ which gives $x_1=1$ and $x_2=0$.
Thus $C^+(g)=\{(1,0)\}$ is a singleton then so is $f(C^+(g))$.
From this, Theorem \ref{theo:rep.plus} shows that there exists $q\in P(g)[x,\bar \lambda]$ with $\bar\lambda=(\lambda_0,\lambda_1,\lambda_2,\lambda_3)$ such that $f-q$ vanishes on $V(h_\text{FJ}^+)$ with $h_\text{FJ}^+$ defined as in \eqref{eq:.polyFJ.plus}, i.e.,
\begin{equation}
\begin{array}{rl}
h_\text{FJ}^+=&(\lambda_0^2\begin{bmatrix}
1\\
0
\end{bmatrix}-\lambda_1^2\begin{bmatrix}
1\\
0
\end{bmatrix}-\lambda_2^2\begin{bmatrix}
0\\
1
\end{bmatrix}-\lambda_3^2\begin{bmatrix}
3(x_1-1)^2\\
-1
\end{bmatrix},\\\\
&\lambda_1^2 x_1,\lambda_2^2x_2,\lambda_3^2 ((x_1-1)^3-x_2),1-\sum_{j=0}^3\lambda_j^2)\,.
\end{array}
\end{equation}
Let $(x,\bar\lambda)\in V(h_\text{FJ}^+)$. 
We get $\lambda_3^2=\lambda_2^2$, so $\lambda_0^2-\lambda_1^2-3\lambda_2^2(x_1-1)^2=0$ and $\lambda_1^2x_1=\lambda_2^2x_2=\lambda_2^2 ((x_1-1)^3-x_2)=0$.
If $\lambda_2=0$, then $\lambda_0^2=\lambda_1^2=\frac{1}{2}$ and $x_1=0$.
If $\lambda_2\ne 0$, then $x_2=0$, $x_1=1$ and $\lambda_1=\lambda_0=0$ which gives $\lambda_2^2=\lambda_3^2=\frac{1}{2}$.
Thus we obtain 
\begin{equation}
V(h_\text{FJ}^+)=\{(0,t,\frac{1}{\sqrt{2}}r_1,\frac{1}{\sqrt{2}}r_2,0,0),(1,0,0,0,\frac{1}{\sqrt{2}}r_1,\frac{1}{\sqrt{2}}r_2)\,:\,r_j=\pm 1\,,\,t\in\R\}\,.
\end{equation}
It is not hart to check that $q=(x_1-1)^3=g_2+g_3\in  Q(g)[x,\bar \lambda]\subset P(g)[x,\bar \lambda]$ and $f-q$ vanishes on $V(h_\text{FJ}^+)$.
\end{example}
\begin{remark}
From Lemma \ref{lem:convex.non-emptyinter} and Examples \ref{exam:infinite.critical},  \ref{exam:nonkkt}, if $C^+(g)$ is finite, then there are cases where $S(g)$ is convex, and there are cases where  $S(g)$ is non-convex.
This implies that the finiteness of $C^+(g)$ does not depend on the convexity of $S(g)$. 
Thus it is still open to find all explicit cases of $g$ where $C^+(g)$ is finite.
\end{remark}

Next, we show in the following example that Theorems \ref{theo:rep.variation}, \ref{theo:rep.plus} and \ref{theo:rep.plus2} are inapplicable to Example \ref{exam:infinite.critical}:

\begin{example}\label{exam:infinite.critical2}
Consider the problem \eqref{eq:pop} with $n=2$, $m=1$, $f=x_1+x_2^2$ and $g=(g_1)=(-x_1^2)$. 
Then the point $(0,0)$ is the unique global minimizer.
The condition $\rank^+(\varphi^g(x))< m$ can be expressed as
\begin{equation}
\exists \lambda_1\in \R_+\backslash\{0\}\,:\,\lambda_1\nabla g_1(x)=\lambda_1\begin{bmatrix}
-2x_1\\
0
\end{bmatrix}=0\,,\,\lambda_1g_1(x)=-\lambda_1 x_1^2=0\,.
\end{equation}
It is equivalent to $x_1=0$. 
Thus $C(g)=C^+(g)=S(g)=\{(0,t)\,:\,t\in\R\}$ is infinite, which implies that 
\begin{equation}
f(C(g))=f(C^+(g))=f(C(g)\cap S(g))=f(C^+(g)\cap S(g))=\{t^2\,:\,t\in\R\}=[0,\infty)
\end{equation}
is infinite.
\end{example}

We prove in the following example that the finiteness assumptions of $f(C^+(g))$ and $f(S(g)\cap C^+(g))$ in Theorems \ref{theo:rep.plus} and \ref{theo:rep.plus2} cannot be removed, respectively:
\begin{example}
Consider problem \eqref{eq:pop} with $n=2$, $m=3$, $f=x_1+x_2$, and $g=(g_1,g_2,g_3)=(x_1^3,x_2^3,-x_1x_2)$. 
It is easy to check that $S(g)=\{(0,t),(t,0)\,:\,t\ge 0\}$ and $f^\star=0$.
Moreover, $(0,0)$ is the unique global minimizer for this problem, and Karush--Kuhn--Tucker conditions do not hold for \eqref{eq:pop} at this point.
The condition $\rank^+(\varphi^g(x))< m$ can be expressed as
\begin{equation}
\begin{cases}
\exists \lambda\in \R_+^3\backslash\{0\}\,:\,\sum_{j=1}^3\lambda_j\nabla g_j(x)=\lambda_1\begin{bmatrix}
2x_1^2\\
0
\end{bmatrix}+\lambda_2\begin{bmatrix}
0\\
2x_2^2
\end{bmatrix}+\lambda_3\begin{bmatrix}
-x_2\\
-x_1
\end{bmatrix}=0\,,\\
\lambda_1g_1(x)=\lambda_1 x_1^3=0\,,\,\lambda_2g_2(x)=\lambda_2 x_2^3=0\,,\,\lambda_3g_3(x)=-\lambda_3 x_1x_2=0\,.
\end{cases}
\end{equation}
It implies that $\lambda_3x_1=\lambda_3x_2=\lambda_1x_1=\lambda_2x_2=0$. 
If $\lambda_3\ne 0$, then $x=0$.
If $\lambda_3=0$, we have either $x_1=0$ or $\lambda_1=0$.
If $\lambda_1=\lambda_3=0$, then $\lambda_2>0$ so $x_2=0$.
Thus $C^+(g)=\{(0,t),(t,0)\,:\,t\in\R\}$ is infinite then so are $f(C^+(g))$ and $f(S(g)\cap C^+(g))$.
Let $h_\text{FJ}^+$ be as in \eqref{eq:.polyFJ.plus}.
Then we get
\begin{equation}
\begin{array}{rl}
h_\text{FJ}^+=&(\lambda_0^2\begin{bmatrix}
1\\
1
\end{bmatrix}-\lambda_1^2\begin{bmatrix}
2x_1^2\\
0
\end{bmatrix}-\lambda_2^2\begin{bmatrix}
0\\
2x_2^2
\end{bmatrix}-\lambda_3^2\begin{bmatrix}
-x_2\\
-x_1
\end{bmatrix},\\\\
&\lambda_1^2 x_1^3,\lambda_2^2x_2^3,-\lambda_3^2 x_1x_2,1-\sum_{j=0}^3\lambda_j^2)\,.
\end{array}
\end{equation}
Let $(x,\bar\lambda)\in V(h_\text{FJ}^+)$.
Then we have $\lambda_1x_1=\lambda_2x_2=\lambda_0^2+\lambda_3^2x_2=\lambda_0^2+\lambda_3^2x_1=0$.
If $\lambda_3=0$, then $\lambda_0=\lambda_1x_1=\lambda_2x_2=0$, which implies $\lambda_1=0$ or $x_1=0$.
If $\lambda_3=\lambda_0=\lambda_1=0$, then $\lambda_2=\pm 1$, which yields $x_2=0$.
If $\lambda_3=\lambda_0=x_1=0$, then $\lambda_2=0$ or $x_2=0$.
If $\lambda_3\ne 0$, then $x_1=x_2=-\lambda_0^2/\lambda_3^2$ and $\lambda_0 \lambda_1=\lambda_0 \lambda_2=0$.
Thus we obtain 
\begin{equation}
\begin{array}{rl}
V(h_\text{FJ}^+)=&\{(t_1,0,0,0,\pm 1,0)\,,\,(0,t_2,0,\pm 1,0,0)\,,\\
&(0,0,0,\cos t_3,\sin t_3,0)\,,\,(-\frac{\mu_0^2}{\mu_3^2},-\frac{\mu_0^2}{\mu_3^2},\mu_0,\mu_1,\mu_2,\mu_3)\,:\\
&t_j\in\R\,,\,\mu_j\in\R\,,\,
\mu_3\ne0\,,\,\sum_{i=0}^3\mu_i^2=1\,,\,\mu_0 \mu_1=\mu_0 \mu_2=0\}\,.
\end{array}
\end{equation}
Let us prove that there does not exist $q\in P(g)[x,\bar \lambda]$ such that $f-q$ vanishes on $V(h_\text{FJ}^+)$.
Assume by contradiction that there is $\sigma_\alpha\in \Sigma^2[x,\bar \lambda]$ such that $f-\sum_{\alpha\in\{0,1\}^3}\sigma_\alpha g^\alpha$ vanishes on $V(h_\text{FJ}^+)$.
Since $(t,0,0,0,1,0)\in V(h_\text{FJ}^+)$ for all $t\in\R$, we get 
\begin{equation}
t-\sigma_{(0,0,0)}(t,0,0,0,1,0) -\sigma_{(1,0,0)}(t,0,0,0,1,0) t^3=0\,.
\end{equation}
Set $\psi_0(t)=\sigma_{(0,0,0)}(t,0,0,0,1,0)\in\Sigma^2[t]$ and $\psi_1(t)=\sigma_{(1,0,0)}(t,0,0,0,1,0)\in\Sigma^2[t]$.
Then it implies that $\psi_0=t(1-t^2\psi_1)$, yielding $\psi_0=t^2\hat \psi_0$ for some $\hat \psi_0\in \Sigma^2[t]$ (since $\psi_0\in\Sigma^2[t]$).
Thus, $t=\psi_0+t^3\psi_1=t^2(\hat \psi_0+t\psi_1)$, and hence this is impossible.
\end{example}
\subsection{Application to exact polynomial optimization}
In this subsection, we apply the representations stated in Section \ref{sec:rep.variation} for computing precisely the optimal value of a polynomial optimization problem using semidefinite programming.
 
We state in the following lemma an extension of the sixth statement of Lemma \ref{lem:mom.sos}:
\begin{lemma}\label{lem:mom.sos.variation}
Let $f,g_1,\dots,g_m,h_1,\dots,h_l\in\R[x]$. 
Let $\bar f^\star$ be as in \eqref{eq:pop.equality} with $g=(g_1,\dots,g_m)$ and $h=(h_1,\dots,h_l)$. 
Assume that there exists $q\in P(g)[x]$ such that $f-\bar f^\star-q$ vanishes on $V(h)\cap S(g)$.
Then there exists $k\in\N$ such that $\rho_k(f,\Pi g,h)=\bar f^\star$.
\end{lemma}
\begin{proof}
The proof is proved similar to the sixth statement of Lemma \ref{lem:mom.sos}.
It is sketched as follows: 
Let $u=f-\bar f^\star-q$.
By assumption,we get  $u=0$ on $S(g)\cap V(h)$. 
From this, Krivine--Stengle's Nichtnegativstellens\"atze \cite{krivine1964anneaux} say that there exists a positive integer $r$ and $w \in P(g)[x]$ such that $u^{2r} + w \in I(h)[x]$.
Let $c=\frac{1}{2r}$. 
Then it holds that $1+t+ct^{2r}\in\Sigma^2[t]$.
Thus for all $\varepsilon>0$, we have
\begin{equation}
f-\bar f^\star+\varepsilon=q + \varepsilon(1+\frac{u}{\varepsilon}+c\left(\frac{u}{\varepsilon}\right)^{2r})-c\varepsilon^{1-2r}(u^{2r} + w) +c\varepsilon^{1-2r}w\in P(g)[x]+I(h)[x]\,.
\end{equation}
Moreover, the degree of the right-hand side has an upper bound independent from $\varepsilon$.
This implies that there exists $k\in\N$ such that for all $\varepsilon>0$, $f-\bar f^\star+\varepsilon \in P_k(g)[x]+I_k(h)[x]$.
Then we for all $\varepsilon>0$, $\bar f^\star-\varepsilon$ is a feasible solution of \eqref{eq:equi.sos} of the value $\rho_k(f,\Pi g,h)$.
It gives $\rho_k(f,\Pi g,h)\ge \bar f^\star-\varepsilon$, for all $\varepsilon>0$, and, in consequence, we get $\rho_k(f,\Pi g,h)\ge \bar f^\star$.
Using the second statement of Lemma \ref{lem:mom.sos}, we obtain that $\rho_k(f,\Pi g,h)= \bar f^\star$, yielding the result.
\end{proof}
\begin{remark}
In Lemma \ref{lem:mom.sos}, if $m\ge 2$, and there exists $q\in Q(g)[x]$ such that $f-\bar f^\star-q$ vanishes on $V(h)\cap S(g)$, we are not sure that $\rho_k(f, g,h)=\bar f^\star$ for some $k\in\N$.
\end{remark}
As a consequence of Lemma \ref{lem:FJ}, the following lemma is obtained:
\begin{lemma}\label{lem:equi.prob.variation}
Let $f,g_1,\dots,g_m\in\R[x]$. 
Let $f^\star$ be as in problem \eqref{eq:pop} with $g=(g_1,\dots,g_m)$.
Let $h_\text{FJ}^+$ be as in \eqref{eq:.polyFJ.plus}.
If problem \eqref{eq:pop} has a global minimizer, then it holds that
\begin{equation}\label{eq:eqi.POP.variation}
\begin{array}{rl}
f^\star=\min\limits_{x,\bar\lambda}& f(x)\\
\text{s.t.}& x\in S(g)\,,\,(x,\bar\lambda)\in V(h_\text{FJ}^+)\,.
\end{array}
\end{equation}
\end{lemma}
We present  in the following two theorems the main application of Theorems \ref{theo:rep.plus} and \ref{theo:rep.plus2} to polynomial optimization:
\begin{theorem}\label{theo:pop.variation}
Let $f,g_1,\dots,g_m\in\R[x]$. 
Let $f^\star$ be as in problem \eqref{eq:pop} with $g=(g_1,\dots,g_m)$.
Let $h_\text{FJ}$ be as in \eqref{eq:.polyFJ} and $h_\text{FJ}^+$ be as in \eqref{eq:.polyFJ.plus}.
Assume that problem \eqref{eq:pop} has at least one global minimizer.
Let $\Pi g$ be as in \eqref{eq:prod.g}.
Then the following statements hold:
\begin{enumerate}
\item If $f(C(g)\cap S(g))$ is finite, there exists $k\in\N$ such that $\rho_k(f,\Pi g,h_\text{FJ})=f^\star$.
\item If $f(C^+(g)\cap S(g))$ is finite, there exists $k\in\N$ such that $\rho_k(f,\Pi g,h_\text{FJ}^+)=f^\star$.
\end{enumerate}
\end{theorem}

\begin{remark}
Since $C^+(g)\cap S(g)\subset C^+(g)$, the second statement of Theorem \ref{theo:pop.variation} holds if $f(C^+(g))$ is finite.
\end{remark}

The proof of Theorem \ref{theo:pop.variation}, which relies on Theorems \ref{theo:rep.variation}, \ref{theo:rep.plus2} and Lemma \ref{lem:mom.sos.variation}, is similar to the one of Theorem \ref{theo:pop}.

\begin{theorem}\label{theo:pop2}
Let $f,g_1,\dots,g_m\in\R[x]$. 
Let $f^\star$ be as in problem \eqref{eq:pop} with $g=(g_1,\dots,g_m)$.
Let $h_\text{FJ}^+$ be as in \eqref{eq:.polyFJ.plus}.
Assume that problem \eqref{eq:pop} has at least one global minimizer and $S(g)$ satisfies the Archimedean condition.
If $f(C^+(g))$ is finite, there exists $k\in\N$ such that $\rho_k(f,g,h_\text{FJ}^+)=f^\star$.
Furthermore, if these exists $R>0$ such that $g_m=R-x_1^2-\dots-x_n^2$, for $k\in\N$ sufficiently large, the Slater condition holds for the SOS relaxation \eqref{eq:sos.relax} of order $k$ with $h=h_\text{FJ}$ or $h=h_\text{FJ}^+$.
\end{theorem}
The proof of Theorem \ref{theo:pop2}, which relies on Theorem \ref{theo:rep.plus} together with the fifth and sixth statements of Lemma \ref{lem:mom.sos}, is similar to the one of Theorem \ref{theo:pop}.

Combining Lemmas  \ref{lem:rep.convex.nomempinter} and  \ref{lem:mom.sos.variation}, we obtain the following corollary:
\begin{corollary}\label{coro:ball.box.simplex.variation}
Let $f,g_1,\dots,g_m\in\R[x]$.
Let $f^\star$ be as in problem \eqref{eq:pop} with $g=(g_1,\dots,g_m)$.
Let $h_\text{FJ}^+$ be as in \eqref{eq:.polyFJ.plus}.
Assume that problem \eqref{eq:pop} has at least one global minimizer, each $g_j$ is concave, and $S(g)$ has a non-empty interior.
Then $S(g)$ is convex and there exists $k\in\N$ such that $\rho_k(f,\Pi g,h_\text{FJ}^+)=f^\star$ with $\Pi g$ being defined as in \eqref{eq:prod.g}.
Moreover, if $S(g)$ satisfies the Archimedean condition, there exists $k\in\N$ such that $\rho_k(f,g,h_\text{FJ}^+)=f^\star$.
\end{corollary}

\subsection{Numerical examples}

In this subsection we report numerical results produced by the SOS relaxations of the values $\rho_k(f,g,0)$ and $\rho_k(f,g,h_\text{FJ}^+)$ for problem \eqref{eq:pop}, where $h_\text{FJ}^+$ is defined as in \eqref{eq:.polyFJ.plus}. 
We use the same notation as in Section \ref{sec:num}.

Our test problem is taken from Example \ref{exam:nonkkt}, namely $n=2$, $m=3$, $f=x_1-1$ and $g=(g_1,g_2,g_3)=(x_1,x_2,(x_1-1)^3-x_2)$. 
It is clear that $f^\star=0$ which is attained at the unique global minimizer $(1,0)$ for problem \eqref{eq:pop}.
Moreover, the Karush--Kuhn--Tucker conditions do not hold at this minimizer.
By using \cite[Proposition 3.4]{nie2014optimality}, we get $\rho_k(f,g,0)< f^\star$ for all $k\in\N$.
As shown in Example \ref{exam:nonkkt}, there exists $q\in  Q(g)[x,\bar \lambda]$ and $f-q$ vanishes on $V(h_\text{FJ}^+)$.
From this, the sixth statement of Theorem \ref{lem:mom.sos} yields that  $\rho_d(f,g,h_\text{FJ}^+)=f^\star$ for some $d\in\N$.
Thus we obtain $\rho_k(f,g,0)<f^\star=\rho_d(f,g,h_\text{FJ}^+)$, for all $k\in\N$.
Since $h_\text{FJ}^+$ has sign symmetry at $\bar \lambda$, we use TSSOS \cite{wang2021tssos} to exploit this structure when computing $\rho_k(f,g,h_\text{FJ}^+)$.

We display the numerical results in Table \ref{tab:bench.variation}.
\begin{table}
\footnotesize
\caption{\footnotesize Lower bounds on $f^\star=0$}
\label{tab:bench.variation}
\begin{center}
\begin{tabular}{ cccc }
 \hline
\multicolumn{4}{c}{$\rho_k(f,g,h_\text{FJ}^+)$} \\
\hline
$k$& value & time & size\\
$3$ &$-0.38154$ & 0.09& (21,186,216,19)\\
$4$&$-0.01199$ & 0.38& (42,476,1056,54)\\
$5$ &$-0.00822$ & 2.07& (84,1110,4393,100)\\
\hline
\multicolumn{4}{c}{$\rho_k(f,g,0)$}\\
\hline
$k$&value & time & size \\
$13$&$-0.02419$& 1.27 & (105,378,1,4)\\
$14$&$-0.02254$& 2.08 & (120,435,1,4)\\
$15$&$-0.02254$& 3.13 & (136,496,1,4)\\
 \hline
\multicolumn{4}{c}{$\rho_k(f,(g,b),0)$} \\
\hline
$k$&value & time & size \\
$13$&$-0.01490$& 1.68 & (105,378,1,5)\\
$14$&$-0.01305$& 2.43 & (120,435,1,5)\\
$15$&$-0.01267$& 3.03 & (136,496,1,5)\\
\end{tabular}
\end{center}
\end{table}
It shows that the numerical value $-0.00822$ of $\rho_5(f,g,h_\text{FJ}^+)$ is the best lower bound on $f^\star$. 
It takes around $2$ seconds to obtain this numerical value. 
It is worth pointing out that the standard SOS relaxations of the values $\rho_k(f,g,0)$ and  $\rho_k(f,(g,b),0)$ cannot reach the bound $-0.00822$ in less than $3$ seconds in spite of using the additional ball constraint $b=2-x_1^2-x_2^2$.
Another observation is that the size of the SOS relaxations of the value $\rho_k(f,g,h_\text{FJ}^+)$ grows faster than the ones of the values $\rho_k(f,g,0)$ and  $\rho_k(f,(g,b),0)$ when $k$ increases.

\section{Representations with denominators}
\label{sec:rep.gen}
We close the paper with  the following Nichtnegativstellens\"atze:
\begin{theorem}\label{theo:rep2}
Let $f,g_1,\dots,g_m\in\R[x]$ such that $f$ is non-negative on $S(g)$  with $g:=(g_1,\dots,g_m)$. 
Then there exists $q\in P(g)[x,\bar \lambda]$ such that $\lambda_0(f-q)$ vanishes on $V(h_\text{FJ})$, where $\bar\lambda:=(\lambda_0,\dots,\lambda_m)$ and $h_\text{FJ}$ is defined as in \eqref{eq:.polyFJ}.
Moreover, if $S(g)$ satisfies the Archimedean condition, we can take $q\in Q(g)[x,\bar \lambda]$.
\end{theorem}
The proof of Theorem \ref{theo:rep2}, which has the same idea as Theorem \ref{theo:rep}, is postponed to below.
Although the representations in Theorem \ref{theo:rep2} have prescribed denominators, we are not required to make any assumption on the image of the set of critical points under $f$.

The following example illustrates the representations stated in Theorem \ref{theo:rep2}:
\begin{example}
Consider Example \ref{exam:infinite.critical} where $V(h_\text{FJ})=\{(0,t,0,\pm 1)\,:\,t\in\R\}$ which gives $\lambda_0(f-0)=0$ on $V(h_\text{FJ})$ since $\lambda_0=0$ on $V(h_\text{FJ})$.
Similar consideration applies to Example \ref{exam:counterexample}.
\end{example}
The following lemma is utilized to prove Theorem \ref{theo:rep2}:
\label{sec:proof.rep2}
\begin{lemma}\label{lem:constant2}
Let $f,g_1,\dots,g_m\in\R[x]$ such that $f(C(g))$ with $g:=(g_1,\dots,g_m)$ is finite.
Let $h_\text{FJ}$ be defined as in \eqref{eq:.polyFJ}.
Let $W$ be a semi-algebraically path connected component of $V(h_\text{FJ})\backslash \{\lambda_0=0\}$. Then $f$ is constant on $W$.
\end{lemma}
\begin{proof}
Choose two arbitrary points $(x^{(0)},\bar \lambda^{(0)})$, $(x^{(1)},\bar \lambda^{(1)})$ in $W$. 
We claim that $f(x^{(0)}) = f(x^{(1)})$.
By assumption, there exists a continuous piecewise-differentiable path $\phi(\tau) = (x(\tau), \bar \lambda(\tau))$, for $\tau\in[0,1]$, lying inside $W$ such that $\phi(0) = (x^{(0)},\bar \lambda^{(0)})$ and $\phi(1) = (x^{(1)},\bar \lambda^{(1)})$ (see, e.g., \cite[Theorem 1.8.1]{pham2016genericity}). 
The Lagrangian function $L(x,\bar \lambda)$ defined in \eqref{eq:Lagran}
is equal to $f(x)$ on $V(h_\text{FJ})\backslash \{\lambda_0=0\}$, which contains $\phi([0,1])$.
By Lemma \ref{lem:composit.semi-al}, the function $L\circ \phi$ is semi-algebraic.
Moreover, the function $L\circ \phi$ is continuous since $L$ and $\phi$ are continuous.
It implies that $L\circ \phi$ is a continuous piecewise-differentiable function thanks to Lemma \ref{lem:semial.func.anal}.
Note that the function
$L\circ \phi$
has zero subgradient on $[0,1]$.
From Lemma \ref{lem:mean.val}, it follows that $f(x (0) )=(L\circ \phi)(0)= (L\circ \phi)(1)= f(x (1) )$.
We now obtain $f (x^{(0)})$ = $f (x^{(1)})$ and hence that $f$ is constant on $W$.
\end{proof}

\subsubsection*{Proof of Theorem \ref{theo:rep2}}
\begin{proof}
Using Lemma \ref{lem:diff.connect}, we decompose $V(h_\text{FJ})\backslash \{\lambda_0=0\}$ into semi-algebraically path connected components:
$Z_1,\dots,Z_s$.
from this, Lemma \ref{lem:constant} yields that $f$ is constant on $Z_i$, which implies that $f(V(h_\text{FJ})\backslash \{\lambda_0=0\})$.
We write
\begin{equation}\label{eq:image.real2}
f(V(h_\text{FJ})\backslash \{\lambda_0=0\}) = \{t_1 ,\dots, t_r \} \subset \R\,,
\end{equation}
where $t_i\ne t_j$ if $i\ne j$.
For $j=1,\dots,r$, set
$W_j:=V_\C(h_\text{FJ},f-t_j)$.
Then $W_j$ is a complex variety defined by finitely many polynomials in $\R[x]$.
We claim that $W_1,\dots,W_r$ are pairwise disjoint. 
Otherwise, let $(x,\bar\lambda)\in W_i\cap W_j$ with $i\ne j$. 
It implies that $t_i=f(x)=t_j$ which is impossible.
Let $U=W_1\cup\dots\cup W_r$.
We now prove that 
\begin{equation}\label{eq:sets.equ}
V(h_\text{FJ})\backslash \{\lambda_0=0\}=(U\backslash \{\lambda_0=0\})\cap \R^{n+m+1}\,.
\end{equation}
Let $(x,\bar\lambda)\in V(h_\text{FJ})\backslash \{\lambda_0=0\}\cap \R^{n}$. 
By \eqref{eq:image.real2}, there exists $j\in\{1,\dots,r\}$ such that $f(x)=t_j$.
It implies that $(x,\bar\lambda)\in W_j\subset U$ and so we get $(x,\bar\lambda)\in U\cap \R^{n+m+1}$. 
Thus $V(h_\text{FJ})\backslash \{\lambda_0=0\}\subset (U \backslash \{\lambda_0=0\})\cap \R^{n+m+1}$ since $(x,\bar\lambda)$ is arbitrary.
Conversely, suppose that $(x,\bar\lambda)\in (U\backslash \{\lambda_0=0\})\cap \R^{n+m+1}$. Then there is $j\in\{1,\dots,r\}$ such that $x\in W_j$.
It implies that $(x,\bar\lambda)\in V(h_\text{FJ})$ by the definition of $W_j$. 
Then $(x,\bar\lambda)\in V(h_\text{FJ})\backslash \{\lambda_0=0\}$.
Thus $(U\backslash \{\lambda_0=0\})\cap \R^{n+m+1}\subset V(h_\text{FJ})\backslash \{\lambda_0=0\}$ since $(x,\bar\lambda)$ is arbitrary.

Let $b=1-\lambda_0^2-\dots-\lambda_m^2$.
By the definition of $U$, Lemma \ref{lem:const.func} shows that there exists $p\in P(g,b)[x,\bar\lambda]$ such that $f-p$ vanishes on $U\cap\R^{n+m+1}$.
We write $p$ as in \eqref{eq:preo.rep.b}
for some $\sigma_\alpha,\psi_\beta\in\Sigma^2[x,\bar{\lambda}]$.
Let 
$q=\sum_{\alpha\in\{0,1\}^m}\sigma_\alpha g^\alpha\in P(g)[x,\bar\lambda]$.
Since $b=0$ on $V(h_\text{FJ})$, $f=p=q$ on $V(h_\text{FJ})\backslash \{\lambda_0=0\}$ thanks to \eqref{eq:sets.equ}.
Thus $\lambda_0(f-q)$ vanishes on $V(h_\text{FJ})$.

Assume that $S(g)$ satisfies the Archimedean condition. Then there exists $R>0$ such that $g_{m+1}=R-x_1^2-\dots-x_n^2\in Q(g)[x]$.
It implies that $S(g,b)$ with $b=1-\lambda_0^2+\dots-\lambda_m^2$ satisfies the Archimedean condition due to \eqref{eq:large.ball}.
By the definition of $U$ Lemma \ref{lem:const.func} yields that there exists $p\in Q(g,b)[x,\bar\lambda]$ such that $f-p$ vanishes on $U\cap\R^{n+m+1}$.
We write $p$ as in \eqref{eq:quadra.rep.b}
for some $\sigma_j\in\Sigma^2[x,\bar{\lambda}]$.
Let 
$q=\sigma_0+\sum_{j=1}^m\sigma_j g_j\in Q(g)[x,\bar\lambda]$.
Since $b=0$ on $V(h_\text{FJ})$, $f=p=q$ on $V(h_\text{FJ})\backslash \{\lambda_0=0\}$ thanks to \eqref{eq:sets.equ}.
Thus $\lambda_0(f-q)$ vanishes on $V(h_\text{FJ})$.
\end{proof}

\paragraph{Acknowledgements.}
The author was supported by the MESRI funding from EDMITT.
\bibliographystyle{abbrv}

\end{document}